\documentclass[11pt,a4paper]{amsart}
\usepackage[utf8]{inputenc}
\usepackage[T1]{fontenc}

\pdfoutput=1

\title{Almost planar finitely presented groups}

\author{John M. Mackay}
\address{School of Mathematics, University of Bristol, Bristol, BS8 1UG, UK.}
\email{john.mackay@bristol.ac.uk}
\author{Joseph P. MacManus}
\address{School of Mathematics,  University of Bristol, Bristol, BS8 1UG, UK, and the Heilbronn Institute for Mathematical Research, Bristol, UK.}
\email{joseph.macmanus@bristol.ac.uk}
\author{Davide Spriano}
\address{Mathematical Institute, University of Warwick, Coventry, CV4 7AL, UK.}
\email{davide.spriano@warwick.ac.uk}
\date{4th May, 2026}

\usepackage{amsfonts}
\usepackage{amsmath}
\usepackage{amsthm}
\usepackage{amssymb}

\usepackage[abbrev,alphabetic]{amsrefs}
\usepackage{stmaryrd}
\usepackage{mathrsfs}  
\usepackage{todonotes}
\usepackage{graphicx}
\usepackage{tikz}
\usepackage{quiver}
\usepackage[colorlinks]{hyperref} 
\hypersetup{
    linkcolor=blue,
    citecolor=red,
}
\usepackage{todonotes}
\usepackage{thmtools,thm-restate}

\DeclareMathOperator{\diam}{diam}
\DeclareMathOperator{\length}{length}

\DeclareMathOperator{\dist}{d}

\DeclareMathOperator{\width}{width}

\newcommand{\bdry}{\partial_\infty}

\newcommand{\R}{\mathbb{R}}
\newcommand{\bbH}{\mathbb{H}}
\newcommand{\Z}{\mathbb{Z}}
\newcommand{\N}{\mathbb{N}}

\newcommand{\into}{\hookrightarrow}

\newcommand{\pres}[2]{\langle #1 \ ; \ #2 \rangle}

\makeatletter
\newcommand{\myitem}[1]{%
\item[#1]\protected@edef\@currentlabel{#1}%
}
\makeatother

\newtheorem{theorem}{Theorem}[section]
\newtheorem*{theorem*}{Theorem}

\newtheorem{proposition}[theorem]{Proposition}
\newtheorem{lemma}[theorem]{Lemma}
\newtheorem{corollary}[theorem]{Corollary}
\newtheorem{claim}[theorem]{Claim}

\newtheorem{question}[theorem]{Question}
\newtheorem{conjecture}[theorem]{Conjecture}

\theoremstyle{definition}
\newtheorem{definition}[theorem]{Definition}

\begin{document}

\begin{abstract}
    We show that finitely presented groups which admit $k$-planar Cayley graphs contain finite-index subgroups with planar Cayley graphs.
	More generally, we answer a question of Georgakopoulos and Papasoglu in the special case of coarsely simply connected graphs: a $k$-planar, coarsely simply connected, connected, locally finite, quasi-transitive graph is quasi-isometric to a planar graph.
    

\end{abstract}

\maketitle




\section{Introduction}

The study and classification of \textit{planar groups}---that is, finitely generated groups which admit planar Cayley graphs---is among the most classical topics in (geometric) group theory, beginning with Maschke's classification of the finite planar groups in 1896 \cite{maschke1896representation}. Infinite planar groups were classified much later, first by work of Wilkie \cite{wilkie1966non} and Zieschang--Vogt--Coldewey \cite{zieschang1970flachen} in the one-ended case. The multi-ended planar groups are more delicate, first studied by Droms \cite{droms2006infinite} and later enumerated by Georgakopoulos--Hamann \cite{georgakopoulos2019planari, georgakopoulos2023planar}, completing their classification.

As is typical in geometric group theory, it is often productive to consider properties `up to finite-index'. A finitely generated group is called \emph{virtually planar} if it contains a finite-index planar subgroup. While planar groups are tricky to classify on-the-nose, the class of virtually planar groups admits an easy and well-known description, as precisely those groups which are virtually free products of free and surface groups. Here, the interesting problem is not to classify these groups but to \emph{characterise} them via their geometry. More generally, one can ask the same about locally finite, quasi-transitive graphs which are quasi-isometric to planar graphs; note that every such graph is quasi-isometric to a planar group \cite{macmanus2024note}, and so morally we should expect such graphs to have analogous characterisations.

In this paper we seek such a characterisation by studying locally finite, quasi-transitive graphs which satisfy a weaker planarity hypothesis. 
Recall that a graph $\Gamma$ is said to be \emph{$k$-planar} if $\Gamma$ can be drawn in the plane so that each edge crosses at most $k$ other edges. The study of such graphs has its roots in work of Ringel \cite{ringel1965sechsfarbenproblem}, where attention was restricted to 1-planar graphs. The topic of $k$-planarity for $k > 1$ received attention from graph theorists much later, for the first time in work of Pach and T\"oth in 1997 \cite{pach1997graphs}.
The class of $k$-planar graphs is famously quite fickle. For example, the problem of deciding whether a given finite graph is even 1-planar is $\mathrm{NP}$-complete \cite{grigoriev2007algorithms}. This is in stark contrast to deciding planarity, which is solvable in linear time \cite{hopcroft1974efficient}. 

We will call  a graph \emph{almost planar} if it is $k$-planar for some $k \geq 0$. 
In the context of groups, almost planarity was first considered by Benjamini and Schramm in 1996 \cite{benjamini1996harmonic}.  Among other things, they show that almost planarity is a quasi-isometry invariant amongst bounded-degree graphs, and so it makes sense to speak of an \emph{almost planar group}, meaning a finitely generated group such that some/every Cayley graph of this group is an almost planar graph.  
The study of almost planarity in graphs and groups has appeared implicitly in the literature in many forms for quite some time. We give a brief survey of this below in Section~\ref{sec:intro-obsrtuctions}. 

The following conjecture, which has, on occasion, been referred to as the \emph{$k$-planar conjecture}, is of interest. It was posed by Georgakopoulos and Papasoglu {\cite[Prob.~9.4]{georgakopoulos2025graph}}; see also \cite[Conj.~6.1]{Esperet_Giocanti_2024}.

\begin{conjecture}[The $k$-planar conjecture]\label{conj:k-planar-conjecture}
    Let $\Gamma$ be a connected, locally finite, quasi-transitive
    graph. Then $\Gamma$ is almost planar if and only if it is quasi-isometric to some planar graph. 
\end{conjecture}

The `if' direction of Conjecture~\ref{conj:k-planar-conjecture} follows from the fact that almost planarity is a quasi-isometry invariant amongst bounded degree graphs. The converse poses an interesting challenge. 
The purpose of this article is to settle this conjecture for a large class of quasi-transitive graphs, namely those graphs which are \emph{coarsely simply connected}. See Definition~\ref{def:csc} for a precise definition of this property.

\begin{restatable}{alphtheorem}{mainresult}\label{thm:main-result}
Let $\Gamma$ be a connected, locally finite, quasi-transitive, coarsely simply connected 
    graph. Then $\Gamma$ is almost planar if and only if it is quasi-isometric to some planar graph. 
\end{restatable}

Most notably, the class of  connected, locally finite, quasi-transitive, coarsely simply connected 
    graphs includes Cayley graphs of finitely presented groups. In particular, we deduce the following corollary.

\begin{restatable}{alphcor}{maincorollary}\label{cor:intro-fpgroups}
    A finitely presented group is almost planar if and only if it is virtually planar.
\end{restatable}

The jump from `quasi-isometric to a planar graph' to `virtually planar' in the above follows from results in \cite{macmanus2023accessibility}, though it is also possible to prove this more directly using our intermediate results in this paper. We give such a direct proof in Section~\ref{sec:upgrade-to-qi} below.

Another source of examples comes from hyperbolic graphs, which are always coarsely simply connected \cite[{Prop.~$\Gamma$.3.23}]{bridson2013metric}.

\begin{restatable}{alphcor}{hypcorollary}\label{cor:intro-hypgroups}
    Every almost planar, hyperbolic, connected, locally finite, quasi-transitive graph is quasi-isometric to a planar graph.
\end{restatable}


Our results add to the ever-growing list of characterisations of virtually planar groups, which is known to be a particularly rigid class of groups. Most famously, it is a deep theorem, originating in the work of Mess \cite{mess1988seifert} and completed with the convergence group theorem of Tukia, Gabai, and Casson--Jungreis \cite{tukia1988homeomorphic, gabai1992convergence, casson1994convergence}, that a finitely generated group which is quasi-isometric to a complete Riemannian plane is virtually planar. This was recently extended to the statement that any finitely generated group which is quasi-isometric to a planar graph (indeed, to a `string graph' \cite{davies2025string}) is virtually planar \cite{macmanus2023accessibility}. For finitely presented groups, this characterisation extends even further to groups which are `asymptotically minor-excluded'  \cite{macmanus2025fat}. 
See also \cite{bowditch2004planar, Esperet_Giocanti_2024, macmanus2025vertex, maillot2001quasi} for more entailments of this form.

The proof of Theorem~\ref{thm:main-result} is outlined in Section~\ref{sec:outline} below.  
Some of our intermediate results apply much more broadly, in turn settling Conjecture~\ref{conj:k-planar-conjecture} for many more graphs and groups. For example, we are able to prove the following statement for one-ended graphs. 

\begin{restatable}{alphtheorem}{mapintohypplane}\label{thm:intro-hyperbolic-plane}
    Let $\Gamma$ be a one-ended, connected, locally finite,  quasi-transitive graph. If $\Gamma$ is almost planar, then one of the following must hold:
    \begin{enumerate}
        \item $\Gamma$ is quasi-isometric to the Euclidean plane.

		\item $\Gamma$ admits a regular map into the hyperbolic plane.
    \end{enumerate}
\end{restatable}

A regular map is coarsely Lipschitz and bounded-to-one in a suitable sense, see Definitions~\ref{def:reg-map-graphs} and \ref{def:reg-map-spaces}.  
Note that a complete solution to Conjecture~\ref{conj:k-planar-conjecture} for one-ended graphs would require upgrading the regular map $f : \Gamma \to \mathbb H^2$ in Theorem~\ref{thm:intro-hyperbolic-plane} to be a quasi-isometry; indeed, even a coarse embedding would suffice: we find such an upgrade when $\Gamma$ is coarsely simply connected, see Section~\ref{sec:outline} below for discussion. 

We remark that Theorem~\ref{thm:intro-hyperbolic-plane} has the immediate consequence that a connected, locally finite, one-ended, almost planar, quasi-transitive graph is either quasi-isometric to the Euclidean plane or has logarithmic separation profile, in the sense of Benjamini--Schramm--Tim\'ar \cite{benjamini2012separation}. 
To the best of our knowledge, it is an open problem to exhibit a connected, locally finite, quasi-transitive graph with logarithmic separation profile which is not hyperbolic. We pose the existence of such a graph as a question below (Question~\ref{question:log-sep}).

\subsection{Known obstructions to almost planarity}\label{sec:intro-obsrtuctions}

We now walk through existing obstructions to almost planarity which appear (sometimes in disguise) in the literature to motivate our approach to this problem. We will not define most of the invariants mentioned and instead defer to references, unless they are directly relevant to this paper. 

One of the most straight-forward ways to obstruct almost planarity is through \emph{asymptotic dimension}. It is known that the asymptotic dimension of any planar graph (in fact, any graph excluding some finite minor) is at most two \cite{bonamy2023asymptotic} (see also \cite{fujiwara2021asymptotic, jorgensen2022geodesic}). Since asymptotic dimension is monotone with respect to regular maps \cite{benjamini2012separation}, we deduce that bounded-degree almost planar graphs have asymptotic dimension at most two.

Another powerful coarse invariant applicable here is the \emph{separation profile} $\mathrm{sep}_\Gamma(n)$ of a graph, introduced by Benjamini--Schramm--Tim\'ar in \cite{benjamini2012separation}. It is a consequence of the famous planar separator theorem of Lipton--Tarjan \cite{lipton1979separator} that a planar graph $\Gamma$ satisfies $\mathrm{sep}_\Gamma(n) = O(\sqrt n)$. If $f : \Gamma \to \Lambda$ is a regular map between bounded-degree graphs, it is known that $\mathrm{sep}_\Gamma(n) = O(\mathrm{sep}_\Lambda(n) )$. In particular, if the separation profile of a bounded-degree graph $\Gamma$ asymptotically dominates $\sqrt n$, then it cannot be almost planar. The separation profile has been computed for many classes of groups. This, in turn, settles Conjecture~\ref{conj:k-planar-conjecture} in many specific cases. For example, almost planar solvable groups must be virtually $\Z^2$ or virtually cyclic \cite{hume2020poincare, gournay2023separation} (in particular, virtually planar). This also implies that $BS(n, m)$ is not almost planar for all $|n| \neq |m|$ \cite{hume2022poincare}. Note that the separation profile \emph{doesn't} obstruct almost planarity for, say, $BS(2,2) \cong F_2 \times \Z$ or hyperbolic groups with conformal dimension less than two, both of which have separation profile $\lesssim \sqrt{n}$ \cite[Lem.~7.2]{benjamini1996harmonic}, \cite{hume2022poincare}.

Separation profiles are generalised by the family of invariants known as \emph{Poincar\'e profiles} $\Lambda_\Gamma^p(r)$, where $p$ ranges over $ [1,\infty)$. This recovers the separation profile via $\mathrm{sep}_\Gamma \simeq \Lambda_\Gamma^1$.  Spielman--Teng's bound on the eigenvalues of the Laplacian for bounded degree planar graphs implies the following, which we record as it is not explicitly in the literature.
\begin{theorem}\label{thm:planar-poincare-profile}
	For any planar graph $\Gamma$ of bounded degree, we have
        \begin{equation*}
                \Lambda_\Gamma^p(r) \lesssim
                \begin{cases}
                        \sqrt{r} & \text{ for } p \in [1,2], \text{ or } \\
                        r^{1-1/p} & \text{ for } p \in [2,\infty),
                \end{cases}
        \end{equation*}
	and this is sharp.
\end{theorem}
\begin{proof}
	The upper bound for $p=2$ is by Spielman--Teng~\cite[Thm.~3.3]{spielman-teng-07-spectral-planar}.
	The bound for $p < 2$ follows from~\cite[Prop.~5]{hume2020poincare}, and for $p >2$ from~\cite[Corollary 2.2]{hume-mackay-25-round-trees}.
 This is sharp when, for example, $X$ is given by attaching a $\mathbb{Z}^2$ quadrant and $3$-regular tree.
\end{proof}
Poincar\'e profiles are also monotone under regular maps, and have also been computed for certain classes of groups. For example, by \cite[Thm.~1.12]{hume2022poincare}, we have for $n > 1$ that $\Lambda_{F_n\times\mathbb{Z}}^2 (r) \simeq r^{2/3} \not\lesssim r^{1/2}$. In particular, this implies that $BS(n,\pm n)$ is almost planar if and only if $n = 1$, which settles Conjecture~\ref{conj:k-planar-conjecture} for Baumslag--Solitar groups. 


In a completely different direction, we can also use the non-existence of certain harmonic functions to obstruct almost planarity. This is due to the work of Benjamini--Schramm. In \cite{benjamini1996harmonic}, they show that if $\Gamma$ is a transient, bounded-degree, connected, almost planar graph, then $\Gamma$ admits a non-constant bounded harmonic function with finite Dirichlet energy. This is a very powerful obstruction, ruling out almost planarity for a vast number of graphs (including $F_2 \times \Z$, see discussion before \cite[Lem.~7.2]{benjamini1996harmonic}). Most dramatically, it is a theorem of Medolla--Soardi that locally finite, quasi-transitive, amenable graphs do not admit such harmonic functions \cite[Thm.~5.2]{medolla1995extension}. Since non-transient locally finite quasi-transitive graphs are either finite or quasi-isometric to either $\Z$ or $\Z^2$ \cite[Thm.~5.13]{woess2000random}, this settles Conjecture~\ref{conj:k-planar-conjecture} for  amenable groups. 
This result also appears in another form in Woess' monograph; see \cite[p.~138]{woess2000random}. Furthermore, when combined with a theorem of Georgakopoulos \cite{georgakopoulos2010lamplighter}, this also implies that many wreath products are not almost planar; in particular the lamplighter group $\Z_2 \wr \Z$ is not almost planar. 

So where does this leave us? Consider the group given by the following presentation: 
$$
G = \pres{a,b,c,d,e,f}{[a,b]=[c,d]=[e,f]}.
$$
The presentation complex of $G$ is depicted in Figure~\ref{fig:bad-group}. 
\begin{figure}[h]
    \centering
    \tikzset{every picture/.style={line width=0.75pt}} 

\begin{tikzpicture}[x=0.75pt,y=0.75pt,yscale=-1,xscale=1]

\draw   (219.33,126.62) .. controls (219.33,108.91) and (251.72,94.56) .. (291.67,94.56) .. controls (331.62,94.56) and (364,108.91) .. (364,126.62) .. controls (364,144.32) and (331.62,158.67) .. (291.67,158.67) .. controls (251.72,158.67) and (219.33,144.32) .. (219.33,126.62) -- cycle ;
\draw    (291.67,94.56) .. controls (302.9,109.02) and (304.7,138.87) .. (291.67,158.67) ;
\draw  [dash pattern={on 0.84pt off 2.51pt}]  (291.67,94.56) .. controls (281.03,109.96) and (277.44,135.41) .. (291.67,158.67) ;
\draw    (291.67,94.56) .. controls (310.99,69.43) and (327.16,59.37) .. (337.34,60) .. controls (347.53,60.63) and (367.59,78.22) .. (291.67,158.67) ;
\draw    (336.48,76.93) .. controls (327.94,81.46) and (323.71,81.43) .. (317.49,78.23) ;
\draw    (332.91,78.41) .. controls (327.79,73.83) and (325.4,74.49) .. (320.8,79.63) ;
\draw    (258.87,123.78) .. controls (250.18,127.98) and (245.95,127.79) .. (239.85,124.35) ;
\draw    (255.25,125.12) .. controls (250.3,120.35) and (247.88,120.92) .. (243.11,125.88) ;
\draw    (348.13,124.72) .. controls (339.43,128.92) and (335.2,128.73) .. (329.11,125.3) ;
\draw    (344.5,126.06) .. controls (339.55,121.29) and (337.14,121.86) .. (332.37,126.82) ;

\end{tikzpicture}
    \caption{}
    \label{fig:bad-group}
\end{figure}

This is a one-ended hyperbolic group which is not virtually Fuchsian. According to Conjecture~\ref{conj:k-planar-conjecture}, it should not be almost planar. However, it slips through the net of all existing obstructions described above; it has asymptotic dimension equal to 2, its Poincar\'e profiles all satisfy the bound of Theorem~\ref{thm:planar-poincare-profile} \cite[Thm.~1.17]{hume-mackay-25-round-trees}, and it admits a non-constant bounded harmonic function with finite Dirichlet energy (as does every non-elementary hyperbolic group \cite{ancona2007positive}). 
It is this particularly troublesome group which motivated the toolbox developed in this paper, and indeed it follows from Corollary~\ref{cor:intro-fpgroups} that this group is not almost planar.

\subsection{Outline of the proof of Theorem~\ref{thm:main-result}}\label{sec:outline}

We now describe the key steps in the proof of our main result. 

The first step is to prove Theorem~\ref{thm:intro-hyperbolic-plane}. To do so, we begin with a regular map $\Gamma\to \Lambda$ where $\Lambda$ is a planar, one-ended, bounded degree graph, and $\Gamma$ is a one-ended, locally finite, quasi-transitive graph; using one-ended-ness and a theorem of Richter--Thomassen \cite{richter20023}, we may assume that the drawing of $\Lambda$ in the plane is VAP-free.\footnote{\emph{VAP-free} stands for \emph{vertex accumulation point-free}, meaning that the image of the vertex set of $\Lambda$ in the plane is discrete.}  We then fill in discs and half-planes in the complement of $\Lambda$ with disks and half-planes from the hyperbolic plane to get a planar complex $P$.  If $P$ is not hyperbolic, this will be seen by bi-Lipschitz embedded loops in $P$ which---having limited interaction with the hyperbolic balls and half-planes---can be pulled back to $\Gamma$ to find subsets whose isoperimetric properties imply (by the Varopoulos inequality) that $\Gamma$ has at most quadratic growth. In particular, this implies that either $\Gamma$ is quasi-isometric to the Euclidean plane, or $P$ is Gromov-hyperbolic. We go on to show that our plane $P$ is in fact quasi-isometric to the hyperbolic plane $\bbH^2$ (see Theorem~\ref{thm:planes-QI-to-hypplane}), which implies Theorem~\ref{thm:intro-hyperbolic-plane}. 

We then proceed with the proof of Theorem~\ref{thm:main-result}. 
Applying results of Hamann on accessible graphs \cite{hamann2018accessibility, hamann2025tree}, which are graph-theoretic generalisations of theorems of Dunwoody \cite{dunwoody1985accessibility} and Papasoglu--Whyte \cite{papasoglu2002quasi}, we immediately reduce to the case that $\Gamma$ is one-ended.
Using Theorem~\ref{thm:intro-hyperbolic-plane}, we obtain that each such graph is either quasi-isometric to the Euclidean plane, or it admits a regular map into the hyperbolic plane. Without loss of generality, we may assume we are in the latter case. To prove Theorem~\ref{thm:main-result}, we need to promote this regular map into the hyperbolic plane to be a quasi-isometry. 
Assuming coarse simply connectedness allows us to replace the graph $\Gamma$ with a 2-dimensional, simply connected, cocompact simplicial complex $X$ admitting a continuous regular map into $\mathbb{H}^2$. 

The core intuitive idea is that we want to map $X$ into $\mathbb{H}^2$ `avoiding folds', essentially meaning that we want to avoid something like the absolute value map $\mathbb{R} \to \mathbb{R}_{\geq 0}$. Morally, such folds are what makes general regular maps far from being quasi-isometries. If one was guaranteed that nothing of the sort would happen, it is believable that a `foldless' regular map $X \to \mathbb{H}^2$ must force $X$ to have a quite restricted geometry. 

Of course, our map $f\colon X\to\mathbb{H}^2$ might not be so well-behaved. Our strategy consists in substituting such map with an improved one. Specifically, there are two properties that we want our improved map to have. The first one is called \emph{width-visibility} (Definition~\ref{defn:widthvisible}), and morally asks for bounded-diameter preimages on a net. To obtain a width-visible map we argue that if the diameter of $f^{-1}(x)$ was not uniformly bounded, we can use transitivity and Arzel\'a--Ascoli to modify the map and reduce a certain notion of complexity. The second one takes inspiration from the methods of Bonk and Kleiner used in \cite{bonk2002rigidity}, and is the notion of \emph{stability} (Definition~\ref{defn:stability}). Roughly, we say that a map is \emph{stable} at a point $x \in X$ if every `modification' $f'$ of $f$ supported in some ball around $x$ still covers $f(x)$. In other words, it is impossible to `push' the image of $f$ so that its image no longer covers $f(x)$. A non-example is depicted in Figure~\ref{fig:stability}. Applying an asymptotic dimension argument together with the Arzel\`a--Ascoli theorem again, we are able to produce another new map which has stable points, and moreover is `very stable': every point is stable for every size of ball.  This latter notion can be interpreted as a strong notion of surjectivity.

We then finish the proof by arguing that a map that is both very stable and width-visible must be a quasi-isometry. In particular, our graph is actually quasi-isometric to $\bbH^2$, which concludes the proof of Theorem~\ref{thm:main-result}.

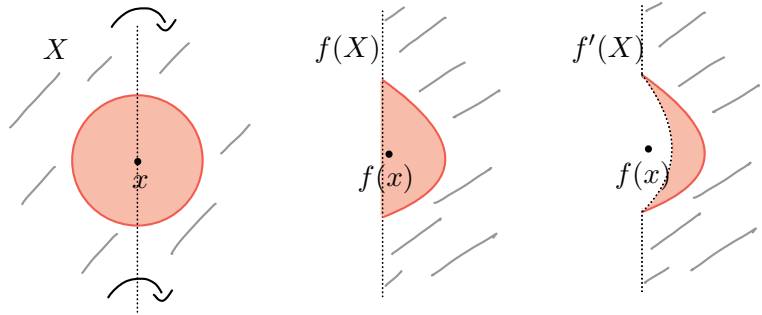
\begin{figure}[ht]
		\begin{tikzpicture}
			\node[anchor=south west,inner sep=0] (image) at (0,0) {\includegraphics[width=0.8\textwidth]{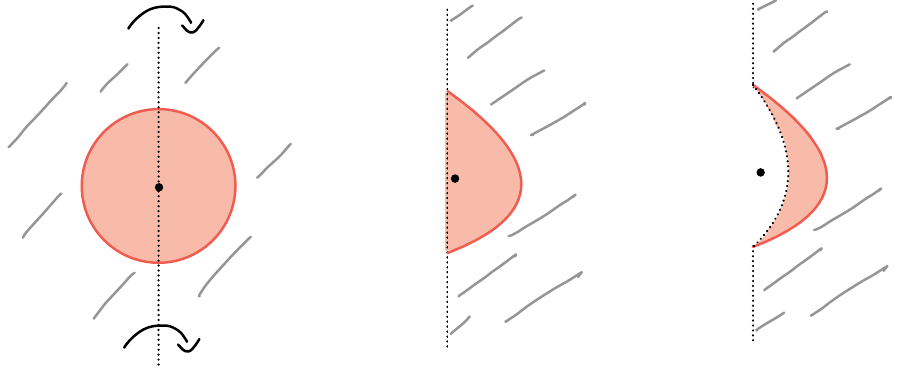}};
			\begin{scope}[x={(image.south east)},y={(image.north west)}]
				\node at (0.18, 0.43) {$x$};
                \node at (0.5, 0.44) {$f(x)$};
                \node at (0.835, 0.45) {$f(x)$};
                \node at (0.07, 0.85) {$X$};
                 \node at (0.45, 0.85) {$f(X)$};
                 \node at (0.79, 0.85) {$f'(X)$};
			\end{scope}
    \end{tikzpicture}\caption{Folds create non-stability: the map $f$, which folds the plane onto itself, is not stable because there is a modification $f'$ that agrees with $f$ outside the red ball and that avoids $f(x)$.}\label{fig:stability}
	\end{figure}

\subsection{Questions}

We conclude this introduction by stating two existential questions. A negative answer to both of these questions would settle Conjecture~\ref{conj:k-planar-conjecture} in its entirety. 

Firstly, as was remarked above, Theorem~\ref{thm:intro-hyperbolic-plane} implies that any almost planar, one-ended, connected, locally finite, quasi-transitive graph $\Gamma$ is either quasi-isometric to the Euclidean plane (in which case, we are very happy), or it has logarithmic separation profile. Currently, all known examples of finitely generated groups with logarithmic separation profile are hyperbolic. Even then, the structure of hyperbolic groups with such a separation profile is quite restrained \cite{lazarovich2025hyperbolic,hume-mackay-25-round-trees}. It is not inconceivable that these could be the only examples; if so, this would settle Conjecture~\ref{conj:k-planar-conjecture} entirely in the one-ended case. We thus pose this as a question (cf. \cite[Question 1.5]{hume2020poorly}).

\begin{question}\label{question:log-sep}
    Does there exist a connected, locally finite, quasi-transitive graph with logarithmic separation profile which is not hyperbolic?
\end{question}

Of course, a positive answer would not be entirely surprising, as the world of non-finitely presented groups is full of paradoxical constructions. For example, it is known that there exists examples of finitely generated groups---in particular, examples of so-called \emph{lacunary hyperbolic groups}---whose separation profile grows arbitrarily slowly on an infinite subset of its domain (though may explode outside of this subset); see \cite[Thm.~1.4]{hume2020poorly}.

There is also the issue of \emph{accessibility} to contend with. There are inaccessible groups which do not contain any one-ended subgroups (for example, see some of the groups constructed in \cite{dunwoody1993inaccessible}). Directly answering the following question is therefore imperative to settling Conjecture~\ref{conj:k-planar-conjecture}.

\begin{question}\label{question:inaccessible}
    Does there exist an inaccessible, connected, locally finite, quasi-transitive graph which is almost planar?
\end{question}

The fact that an inaccessible transitive graph cannot be quasi-isometric to a planar graph was the main technical result of \cite{macmanus2023accessibility}, and it is likely that resolution to Question~\ref{question:inaccessible} would involve reducing to this case. It is also known---a fact which is due to Esperet, Giocanti, and Legrand-Duchesne---that they cannot be minor-excluded \cite{esperet2024structure}. Outside of these results, our current understanding of the geometry of inaccessible transitive graphs is quite limited.

\subsection*{Acknowledgements}

The first author would like to thank the Isaac Newton Institute for Mathematical Sciences, Cambridge, for support and hospitality during the programme ``Operators, Graphs, Groups'', where work on this paper was undertaken; this work was supported by EPSRC grant EP/Z000580/1.
He also thanks David Hume, Romain Tessera and, especially, Mario Bonk for discussions around Theorem~\ref{thm:planar-poincare-profile}.
The second author was supported by the Additional Funding Programme for Mathematical Sciences, delivered by EPSRC (EP/V521917/1) and the Heilbronn Institute for Mathematical Research. The third author was supported by a Royal Society University Research Fellowship  URF\textbackslash R1\textbackslash 251707, and he would like to thank Federico Vigolo for discussions on an early version of this paper.  


\section{Preliminaries}

In this section we establish our notational and terminological conventions.



\subsection{Coarse geometry}
In this paper we will often consider several metric spaces at the same time, to avoid confusion we will use subscripts to denote the metric space we are referring to. For instance, $\dist_X$ refers to the distance in the metric space $X$. The \emph{(closed) neighbourhood} of a set $A$ is the set $B_X(A, r)$ of points in $X$ at distance $r$ or less from $A$. We may sometimes suppress the decorative subscripts from our notation if the space under consideration is clear from context.

A map $f$ between metric spaces $X,Y$ is \emph{$(\kappa, \varepsilon)$--quasi-isometric embedding} if for any pair of points $x,y\in X$ it holds \[\frac{1}{\kappa} \dist_X (x,y) - \varepsilon \leq^{(1)} \dist_Y(f(x),f(y)) \leq^{(2)} \kappa \dist_X(x,y) + \varepsilon.\]
If only the inequality $(2)$ holds, we say that $f$ is \emph{$(\kappa, \varepsilon)$--coarsely Lipschitz}. A map is $\kappa$--bi-Lipschitz if it is a $(\kappa, 0)$--quasi-isometric embedding. 
A map is a \emph{$(\kappa, \varepsilon)$--quasi-isometry} if it is a $(\kappa, \varepsilon)$--quasi-isometric embedding and it is $\varepsilon$--coarsely surjective, meaning $B_Y(f(X), \varepsilon) = Y$.

A metric space $X$ is said to have \emph{bounded geometry} if for all $r \geq 0$, there exists $n > 0$ such that every closed ball of radius $r$ in $X$ can be covered by $n$ balls of radius 1. Note that if $X$ is a graph, this is equivalent to asking that $X$ have bounded degree. 

The following definition will play a central role.

\begin{definition}[Coarsely simply connected]\label{def:csc}
    Let $\Gamma$ be a graph and $m\geq 3$ an integer. The \emph{$m$-filling} of $\Gamma$, denoted $X_m(\Gamma)$, is the polygonal 2-complex obtained by attaching 2-cells to all cycles of length at most $m$. 
    We say that  $\Gamma$ is \emph{coarsely simply connected} if there exists $m \geq 3$ such that $X_m(\Gamma)$ is simply connected.
\end{definition}

The motivating examples include Cayley graphs of finitely presented groups, as well as hyperbolic graphs.

\subsection{Regular maps}\label{ssec:reg-maps}

We now define the notion of a regular map. We will give two definitions, one for graphs and the other for generic metric spaces. 

\begin{definition}[Regular maps between graphs \cite{benjamini2012separation}]\label{def:reg-map-graphs}
    Let $\Gamma$, $\Lambda$ be graphs, and $\kappa \geq 1$ a constant. Then, a map $f : V(\Gamma) \to V(\Lambda)$ is said to be \emph{$\kappa$-regular} if the following hold:
\begin{enumerate}
    \item The map $f$ is $(\kappa,\kappa)$-coarsely Lipschitz. 
    \item For all $y \in V(\Lambda)$, we have that $|f^{-1}(y)| \leq \kappa$. 
\end{enumerate}
\end{definition}

We will often abuse notation and write $f : \Gamma \to \Lambda$, meaning $f : V(\Gamma) \to V(\Lambda)$.

\begin{definition}[Regular maps between metric spaces {\cite[Def.~1.3]{benjamini1996harmonic}}]\label{def:reg-map-spaces}
    Let $X$, $Y$ be metric spaces. 
    Then, a map $f : X \to Y$ is said to be \emph{$\kappa$-regular} if the following hold:
\begin{enumerate}
    \item The map $f$ is $(\kappa,\kappa)$-coarsely Lipschitz. 
    \item For all $y \in Y$, we have that $f^{-1}(B_Y(y,1))$ can be covered by $\kappa$ balls of radius $1$. 
\end{enumerate}
\end{definition}

Note that we do not require $f$ to be continuous in the above definition. However, continuity will often be desirable, and so we may speak of \emph{continuous regular maps}. 

Regular maps arise naturally from subgraph inclusion, and from coarse embeddings (e.g.\ quasi-isometric embeddings) of graphs.  For example, if $G$ and $H$ are finitely generated groups and $H$ includes into $G$ as a subgroup, then this inclusion induces a regular map between the Cayley graphs of $G$ and $H$.

\subsection{Almost planar graphs and groups}\label{sec:prelims-almost-planar}

Recall that a graph is said to be \emph{$k$-planar} if it can be drawn in the plane such that every edge crosses at most $k$ other edges. For example, 0-planar graphs coincide with planar graphs. We say that a graph is \emph{almost planar} if it is $k$-planar for some finite $k \geq 0$. The minimal $k$ such that a graph is $k$-planar is sometimes called the \emph{local crossing number} of the graph.

The following is due to Benjamini and Schramm \cite{benjamini1996harmonic} (in fact, it is their original definition of almost planarity), and plays a central role throughout this paper.

\begin{theorem}[{\cite[Thm.~1.8]{benjamini1996harmonic}}]\label{prop:almost-planar-regular-map-char}
    Let $\Gamma$ be a bounded-degree, connected graph. Then $\Gamma$ is almost planar if and only if there exists a planar, bounded-degree, connected graph $\Lambda$ and a regular map $\rho : \Gamma \to \Lambda$. 
\end{theorem}

This has the following immediate consequence. 

\begin{corollary}
    Let $\Gamma$, $\Lambda$ be connected, bounded degree graphs.
    If $f : \Gamma \to \Lambda$ is a regular map and $\Lambda$ is almost planar, then $\Gamma$ is almost planar.
\end{corollary}

In particular, almost planarity is a quasi-isometry invariant among connected graphs of bounded degree. This allows us to meaningfully describe a finitely generated group as an \emph{almost planar group}, meaning some/all of its Cayley graphs are almost planar graphs.
Then, another consequence of Theorem~\ref{prop:almost-planar-regular-map-char} is the following.

\begin{corollary}
    Let $G$ be a finitely generated group and $H \leq G$ a finitely generated subgroup. If $G$ is almost planar, then so is $H$.
\end{corollary}

\subsection{Hyperbolicity and boundaries}
We say that a geodesic metric space $X$ is \emph{$\delta$--hyperbolic} if for any geodesic triangle with sides $\alpha, \beta, \gamma$ the side $\gamma$ is contained in the $\delta$--neighbourhood of $\alpha\cup \beta$. A space is \emph{hyperbolic} if it is $\delta$--hyperbolic for some $\delta$. A group is hyperbolic if one (hence any) of its Cayley graphs with respect to a finite generating set is hyperbolic. 

For a hyperbolic space $X$, the \emph{Gromov boundary} $\partial_\infty X$ is the set of geodesic rays up to bounded Hausdorff distance. If $X$ is a space of bounded geometry, then $\partial_\infty X$ is a compact metrizable space. Given points $x, y, o \in X$, the \emph{Gromov product} $(x \mid y)_o$ is the quantity $\frac{1}{2}\left( \dist_X (o, x) + \dist_X(o, y) - \dist_X(x,y)\right)$. Given $a,b \in \bdry X$ we define \[ (a | b)_o = \sup \liminf_{t \to \infty} (\gamma(t)|\eta(t))_o \] where the supremum is over geodesic rays $\gamma, \eta$ based at $o$ representing $a,b$ respectively.  It is sometimes convenient to denote a geodesic ray from $o$ representing $a \in \bdry X$ as $[o,a)$.  A \emph{visual metric} is a metric $\rho$ on $\partial_\infty X$  such that there are $\varepsilon, C>0$ so that for any $a, b \in \partial_\infty X$ it holds $\frac{1}{C}e^{-\varepsilon (a \vert b)_o} \leq \rho(a, b) \leq C e^{\varepsilon (a \vert b)_o}$. The topology induced by any visual metric coincides with the Gromov topology.

\section{Proof of Theorem~\ref{thm:intro-hyperbolic-plane}}

In this section we prove Theorem~\ref{thm:intro-hyperbolic-plane} of the introduction. That is, if a one-ended, locally finite, quasi-transitive graph $\Gamma$ is almost planar, then either $\Gamma$ is quasi-isometric to $\R^2$, or maps regularly into $\bbH^2$. We prove this in several steps.

\subsection{Gromov-hyperbolic planes}

We begin by observing that the hyperbolic plane is essentially the \emph{unique} visual, Gromov-hyperbolic, complete Riemannian plane, up to quasi-isometry.
First, we will need the following definition.

\begin{definition}[Visual space]
    Let $X$ be a geodesic metric space, $K \geq 0$. We say that $X$ is \emph{$K$-visual} if for every $x, y \in X$, there exists a geodesic ray $\gamma$ based at $x$ such that 
    $$
    \dist_X(y,\gamma) \leq K.
    $$
\end{definition}

We now prove the following.

\begin{theorem}\label{thm:planes-QI-to-hypplane}
    Let $P$ be a Gromov-hyperbolic, visual, complete Riemannian plane of bounded geometry. Then $P$ is quasi-isometric to the hyperbolic plane $\bbH^2$. 
\end{theorem}

\begin{proof}
	Let $\bdry P$ be the boundary at infinity of $P$ with some choice of visual metric $\rho$.
	
	We now assemble some results from the literature, omitting definitions which we do not work with directly.
	By Bonk--Schramm~\cite[Thms.~7.4,~8.2]{bonk2000embeddings}, $P$ is quasi-isometric to $\bbH^2$ if and only if $\bdry P$ and $\bdry \bbH^2 = \mathbb{S}^1$ are ``power quasisymmetric''.
	For connected metric spaces, being power quasisymmetric is equivalent to being quasisymmetric~\cite[Cor.~3.12]{TukiaVaisala-80-qs} (the ``$\lambda$-HD'' definition there is weaker than being connected).
	And again by Tukia--V\"ais\"al\"a~\cite[Thm.~4.9]{TukiaVaisala-80-qs}, a topological circle $Z$ is quasisymmetric to $\mathbb{S}^1$ if and only if it is ``HTB'' and ``BT''.  

	The ``HTB'' definition is equivalent to being ``doubling'', which holds for $\bdry P$ by the bounded geometry assumption~\cite[Thm. 9.2]{bonk2000embeddings}.

	The ``BT'' or ``bounded turning'' condition is also known as being ``linearly connected'': a metric space $Z$ is \emph{linearly connected} if there exists $L \geq 1$ so that for any $a,b \in Z$ there exists a connected set $J$ containing $a,b$ such that $\diam(J) \leq L \dist_Z(a,b)$.

	Summarising the discussion so far: to show $P$ is quasi-isometric to $\bbH^2$ it remains to show $\bdry P$ is a linearly connected topological circle.

	Moore showed that a topological space $Z$ is homeomorphic to $\mathbb{S}^1$ if and only if it is metrizable, connected, compact and has the property that for any distinct $x,y \in Z$, $Z \setminus \{x,y\}$ is disconnected (see e.g.~\cite[Thm.~2.28]{Hocking-Young-Topology}).

	Fix a basepoint $o\in P$ and hyperbolicity constant $\delta$ for $P$.
	Since $P$ is proper, $\bdry P$ is compact: any sequence of points $(a_i) \subset \bdry P$ has a convergent subsequence, since one can apply Arzel\`a--Ascoli to the sequence of geodesic rays $[o,a_i)$.

	We now see that $\bdry P$ is connected.  Assume the visual metric $\rho$ on $\bdry P$ satisfies 
	\[
		\frac{1}{C} e^{-\epsilon (a|b)_o} \leq \rho(a,b) \leq C e^{-\epsilon (a|b)_o}, \ \text{ for all } a,b \in \bdry P.
	\]
	For any large $R$ since $P$ is a Riemannian plane, there exists a topological circle $\gamma \subset P \setminus B(o,R)$ which disconnects $B(o,R)$ from infinity.
	For any $a,b \in \bdry P$, we can find geodesic rays $[o,a), [o,b)$ which hit $\gamma$ at some points $p,q$.  As $\gamma$ is connected, we can find a sequence of points $p=p_0, p_1,\ldots, p_n=q$ in $\gamma$ so that $\dist_P(p_{i-1},p_i) \leq 1$ for $i=1,\ldots,n$.  Since $P$ is visual, there exist geodesic rays $[o,a_i), i=0,\ldots,n$, going to points $a=a_0,a_1,\ldots,a_n=b$ in $\bdry P$, and so that $\dist_P(p_i,[o,a_i)) \leq K$.
	By hyperbolicity, $(a_i|a_{i+1})_o \geq \dist_P(o,p_i)-C' \geq R-C'$, where $C' = C'(\delta, K)$, and so $\rho(a_i,a_{i+1}) \leq C e^{-\epsilon (R-C')}$.
	Since $R$ was arbitrary, we conclude that $\bdry P$ is connected.

	We now show that removing any distinct $a,b \in \bdry P$ disconnects $\bdry P$.
	Since $P$ is proper and Gromov-hyperbolic, there exists a geodesic $\gamma:\R\to P$ so that $\gamma(-\infty)=a$ and $\gamma(+\infty)=b$.
	By the Jordan curve theorem, $\gamma$ disconnects $P$ into two components, $U$ and $U'$.
	For any $c \in \bdry P \setminus \{a,b\}$, any geodesic ray $[o,c)$, and any $C>0$, the set $B_P([o,c),C)\cap \gamma$ is bounded, else by hyperbolicity we would have $c=a$ or $c=b$.  In particular, which of $U$ or $U'$ the ray $[o,c)$ ends up in is independent of the choice of ray.

	Let $W,W' \subset \bdry P$ be the set of points $c \in \bdry P \setminus \{a,b\}$ for which the geodesic $[o,c)$ ends up in $U, U'$ respectively.
	By the above argument $\bdry P \setminus \{a,b\}$ is the disjoint union of $W$ and $W'$.
	Since $\gamma$ is a geodesic, $W$ and $W'$ are open because for each point of $W$ represented by $\xi$, a neighbourhood for $\xi$ contained in $W$ is obtained by considering all geodesics that diverge from $\xi$ after it last was $(2\delta+1)$-close to $\gamma$.
	Both $W$ and $W'$ are non-empty: choose $R$ much larger than $\delta$ and $\dist_P(o,\gamma)$.  Then the topological circle outside $B(o,R)$ contains points which are in both $U$ and $U'$, and a large (compared to hyperbolicity) distance from $\gamma$, so gives limit points in both $W$ and $W'$.
	So we have shown $\bdry P$ is a topological circle.  In fact, consequently we have seen that $W,W'$ are connected open intervals in $\bdry P$.

	It remains to show $\bdry P$ is linearly connected.
	Take $a,b \in \bdry P$, connected by a geodesic $\gamma \subset P$ at least say $10\delta$ from $o$ (else it suffices to bound the diameter of an arc connecting $a,b$ by $\diam(\bdry P)$).
	Suppose $U$ is the component of $P \setminus \gamma$ not containing $o$ and let $W$ be defined as before.
	Then $W \cup \{a,b\}$ is a connected set containing $a,b$, and any point $c \subset W$ is the limit of a geodesic ray which crosses $\gamma$, hence for suitable $C',C''$, $(a|c)_o \geq \dist_P(o,\gamma)-C' \geq (a|b)_o -C''$, thus for 
	\[
		\diam(W) \leq 2\sup_{c \in W} \rho(a,c) \leq 2e^{\epsilon C''}C^2 \rho(a,b). \qedhere
	\]
\end{proof}

\subsection{A criterion for visual-ness}

In order to apply Theorem~\ref{thm:planes-QI-to-hypplane}, we will need an easy way of establishing that a given Gromov-hyperbolic plane is visual. The following proposition provides such a criterion. 

\begin{lemma}\label{prop:criterion-for-visibility}
	Let $P$ be a $\delta$-hyperbolic, simplicial plane of bounded geometry. Suppose that there exists $C > 0$ such that every combinatorial Jordan curve of length at most $12\delta$ bounds a disk containing at most $C$ vertices. Then $P$ is visual. 
\end{lemma}

\begin{proof}
    We must show that there exists $K > 0$ such that for every $x, y \in P$, there exists a geodesic ray $\gamma$ based at $x$ such that $\dist_P(y,\gamma) < K$. 
    To do so, we will show that `balls are uniformly visual' and then apply the Arzel\`a--Ascoli theorem to conclude. 
    
	Let $r > C + 20 \delta$ and $x \in P$ be arbitrary. Let $S_r$ denote a simple closed curve in the 1-skeleton of $P$ such that 
    \begin{enumerate}
        \item the point $x$ lies in a bounded region of $P \setminus S_r$, and 

		\item for all vertices $s \in S_r$, we have that $\dist_P(s,x) = r$. 
    \end{enumerate}
	Such a simple closed curve $S_r$ certainly exists for any such $r\in\N$. 
    By the Jordan curve theorem, let $D_r$ denote the disk bound by $S_r$.  

    \begin{claim}\label{claim:spheres-are-visual}
        For all $y \in D_r$, there exists $z \in S_r$ and a geodesic $\gamma$ from $x$ to $z$ such that 
        $$
        \dist_P(y,\gamma) < C + 6\delta. 
        $$
    \end{claim}

    \begin{proof}
		For each vertex $z$ of $S_r$ choose once and for all a geodesic $\eta_z$ from $z$ to $x$ so that the choices form a tree, meaning that if $z,w\in S_r$ and $\eta_z(t) =\eta_w(t)$, then $\eta_z(s)  =\eta_w(s)$ for all $s\geq t$. In particular, either $y$ is already on one such geodesic, in which case we are done, or there exists a unique pair $z, w$ so that $y$ is in the interior of the triangle with sides $\eta_z, \eta_w$ and the edge $(z,w)$. Let $\Delta$ be the closure of such inside.  For $ n \leq \frac{r}{4\delta}$ define $z_n = \eta_z(n 4\delta)$, and define $w_n$ similarly (note that $\delta>0$). By hyperbolicity all triangles are uniformly $\delta$--thin. Thus we have $\dist_P(z_n, w_n) \leq  2\delta$ for all $n$. Additionally, for all $n$ choose once and for all a geodesic $[z_n, w_n]$ between $z_n, w_n$ that has connected intersection with both $\eta_z$ and $\eta_w$. The latter can be assumed by replacing the initial/final segment of $[z_n, w_n]$ with the appropriate segment of $\eta_z$/$\eta_w$. Moreover, for $n \neq m$ we must have $[z_n, w_n] \cap [z_m, w_m] = \emptyset$. If not, a triangular inequality argument would yield $\dist_P(z_n, z_m) < 4\delta$. A similar triangular inequality argument shows that $[z_n, w_n]$ cannot intersect the sphere of radius $ r- n4\delta + 2 \delta$ around $x$.    
        
    Intuitively, we want to use the geodesics $[z_n, w_n]$ to chop up  $\Delta$ in pieces of bounded perimeter, and conclude that $y$ must be contained in one such piece. For this, consider the quadrilaterals $Q_n$ with vertices $z_n, z_{n+1}, w_n, w_{n+1}$ and edges $[z_n, w_n], [z_{n+1},w_{n+1}]$ and the appropriate restrictions of $\eta_w, \eta_z$. If $Q_0$ contains $x$ then it must contain $\eta_z, \eta_w$ and the whole $\Delta$. So the claim is satisfied. So assume it is not the case. Consider $Q_1$. If $Q_1$ contains $y$, we are done. If $Q_1$ contains $x$, arguing as before we have that $Q_0 \cup Q_1$ contain the whole $\Delta$, and hence $Q_1$ contains $y$. If neither of the above happens, let $A_{n}$ denote complement of the ball $B_P(x, r-n4\delta)$. Since $[z_2, w_2]$ does not intersect the sphere of radius $r - 6\delta$ around $x$ we have 
    \[\Delta \cap  A_1 \subseteq Q_0 \cup Q_1.\]
    Proceeding inductively we have that for each $n$, either $y\in Q_n$ and we stop, or it must hold \[\Delta \cap A_n \subseteq \bigcup_{i =0}^n Q_i.\]
		If we don't stop, we eventually conclude $y \in B_P(x, 6\delta)$, and the hypothesis holds with $\gamma$ being any geodesic from $x$ to $S_r$. Otherwise, $y\in Q_n$ for some $n$. By construction, the perimeter of $Q_n$ is at most $12\delta$. Thus, the distance between $y$ and a point in $Q_n$ is at most $C$, otherwise $Q_n$ would contain a path containing more than $C$ vertices. Thus, the distance between $y$ and any point of $Q_n$ is at most $C+6\delta$. Since $Q_n$ contains points of the geodesic $\eta_w$, the claim follows for $\gamma = \eta_w$. See Figure~\ref{fig:visibility-proof} for a cartoon of this proof.
    \end{proof}

    \begin{figure}[h]
        \centering
        \input{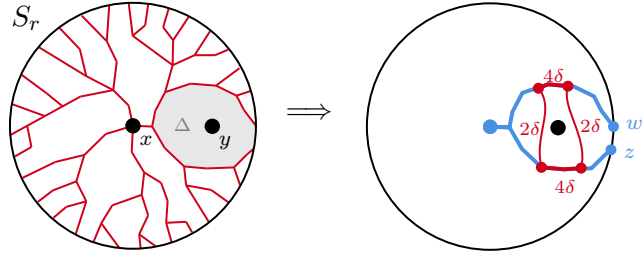}
        \caption{Proving that spheres are uniformly visual in $P$.}
        \label{fig:visibility-proof}
    \end{figure}

    In light of Claim~\ref{claim:spheres-are-visual}, let $K = C + 6\delta$.
    We have that for all $x, y \in P$ and all $r > 0$, there exists a geodesic segment $\gamma$ based at $x$ of length $r$ which passes $K$-close to $y$. By the Arzel\`a--Ascoli theorem, there exists a geodesic ray based at $x$ which passes within $K$ of $y$. Since $x$ and $y$ were arbitrary, it follows that $P$ is visual. 
\end{proof}

\subsection{The Varopoulos inequality}

We will make use of an isoperimetric inequality of Varopoulos, as it appears in \cite[Thm.~2.1]{saloff1995isoperimetric}.
Given a vertex-transitive graph $\Gamma$, let $\gamma_\Gamma(n)$ denote the growth function of $\Gamma$; that is, the number of vertices in any $n$-ball. Given a subset $\Omega \subset V(\Gamma)$, we denote by $\partial \Omega$ the set of vertices which lie at a distance of exactly 1 from $\Omega$. 

\begin{theorem}[Varopoulos inequality]
	Any connected, locally finite, vertex-transitive graph $\Gamma$ satisfies the following:
    
	For all finite  $\Omega \subset V(\Gamma)$, we have that 
	$$
	\frac{|\Omega|}{\phi(2|\Omega|)} \leq 4|\partial\Omega|,
	$$
	where $\phi(\lambda) = \inf\{n \in \N : \gamma_\Gamma(n) > \lambda\}$. 
\end{theorem}

The statement above has the easy consequence that if there exists $C, D \in \N$ such that $\gamma_\Gamma(n) \geq Cn^D$ for all $n \geq 1$, then there exists $C'> 0$ such that 
$$
|\Omega|^{\tfrac{D-1}{D}} \leq C' |\partial \Omega|
$$
for all finite  $\Omega \subset V(\Gamma)$.
The contrapositive of this says that if we can find `large' subsets with a `small' boundary, then this subsequently bounds the growth of $\Gamma$. 
To make this precise, given a non-empty finite subset $\Omega \subset V(\Gamma)$, define the \emph{depth} of $\Omega$, denoted $\mathrm{depth}(\Omega)$, via
$$
\mathrm{depth}(\Omega) := \max \{\dist_\Gamma(x, \partial \Omega) : x \in \Omega\}.
$$
We then have the following corollary.

\begin{corollary}\label{cor:varopoulos}
    Let $\Gamma$ be a connected, locally finite, vertex-transitive graph. Let $(\Omega_n)$ be a sequence of finite subsets of $V(\Gamma)$, satisfying $|\Omega_n| \to \infty$. \begin{enumerate}
        \item\label{itm:varop-cor-1} If $|\partial \Omega_n|$ is uniformly bounded, then $\Gamma$ has at most linear growth. 

        \item\label{itm:varop-cor-2} If $$
        \frac{\mathrm{depth}(\Omega_n)}{|\partial \Omega_n|} \not\to 0,
        $$ then $\Gamma$ has at most quadratic growth. 
    \end{enumerate}
\end{corollary}

Indeed, to see (\ref{itm:varop-cor-2}) we simply observe that if $\gamma_\Gamma(n) \geq C n^D$ and $\mathrm{depth}(\Omega_n) \geq c |\partial \Omega_n|$, then
\[
	C' |\partial \Omega_n| \geq |\Omega_n|^{\frac{D-1}{D}} \geq \left(C\mathrm{depth}(\Omega_n)^D \right)^{\frac{D-1}{D}}
	\geq (C^{1/D}c)^{D-1} |\partial \Omega_n|^{D-1},
\]
implying $D \leq 2$ by taking $n \to \infty$.

The reader with a careful eye might have noticed that these results apply to \emph{vertex-transitive} graphs rather than \emph{quasi-transitive} graphs. This is not a problem for our purposes, as every connected, locally finite, quasi-transitive graph is quasi-isometric to a connected, locally finite, vertex-transitive graph \cite[Prop.~5.1]{woess1994topological}, namely the \emph{Cayley--Abels graph} of its isometry group. With this, most applications of the Varopoulos inequality will pass through this quasi-isometry unhindered. 
We briefly rephrase (\ref{itm:varop-cor-1}) for easy reference later on. For this, we introduce a definition.

\begin{definition}
    Let $\Gamma$ be a connected graph. Then $\Gamma$ has property (\ref{eq:small-cuts}) if the following holds:
    \begin{equation}\tag{$\dagger$}\label{eq:small-cuts}
        \parbox{4.5in}{For all $n \geq 0$ there exists $m \geq 0$ such that if $S \subset V(\Gamma)$ satisfies $|S| \leq n$, then every finite connected component of $\Gamma \setminus S$ contains at most $m$ vertices.}
    \end{equation}
\end{definition}

\begin{proposition}\label{lem:finite-components-in-fg-groups}
    Let $\Gamma$ be an infinite, connected, locally finite, quasi-transitive graph which is not two-ended. Then $\Gamma$ satisfies property (\ref{eq:small-cuts}). 
\end{proposition}

 Indeed, if $\Omega$ is a component of $\Gamma \setminus S$ then $\partial \Omega \subset S$, so if $\Gamma$ is not two-ended, by Corollary~\ref{cor:varopoulos}(\ref{itm:varop-cor-1}) there is a bound on the size of finite such $\Omega$ depending only on $S$.


\subsection{Constructing a plane}

We show how to upgrade a regular map into a planar graph to a regular map into a Gromov-hyperbolic plane, under suitable hypotheses. 
First, we will need the following two easy lemmas. 

\begin{lemma}\label{lem:one-ended-regular-maps}
	Suppose $f : X \to Y$ is a regular map between proper geodesic metric spaces of bounded geometry. Assume that $X$ is one-ended, and $f$ is coarsely surjective. Then $Y$ is also one-ended. 
\end{lemma}

\begin{proof}
    Suppose $Y$ were not one-ended, so there exists a compact set $S \subset Y$ such that $Y\setminus S$ contains two unbounded components. Let $S' = f^{-1}(S)$. Since $f$ is regular and $X$ has bounded geometry, $S'$ is covered by a finite number of compact balls, and in particular is compact itself. Since $f$ is coarsely surjective and $X$ has bounded geometry it is easy to see that $S'$ must separate two unbounded components. But then $X$ has more than one end, which is a contradiction. 
\end{proof}

\begin{lemma}\label{lem:small-cuts-regular-maps}
    Suppose $f : \Gamma \to \Lambda$ is a coarsely surjective, regular map between connected, bounded-degree graphs. Suppose further that $\Gamma$ satisfies (\ref{eq:small-cuts}).
    Then $\Lambda$ also satisfies (\ref{eq:small-cuts}).
\end{lemma}

\begin{proof}
    Follows from a similar argument to the proof of Lemma~\ref{lem:one-ended-regular-maps}. 
\end{proof}

We are now able to prove the following key intermediate step.

\begin{proposition}\label{prop:map-into-gromov-hyp-plane}
    Let $\Gamma$ be connected, locally finite, one-ended, almost planar, vertex-transitive graph. Then either $\Gamma$ has quadratic growth or $\Gamma$ admits a regular map to the 1-skeleton of a connected simplicial complex of bounded geometry, which is Gromov-hyperbolic, visual, and homeomorphic to the plane.
\end{proposition}

\begin{proof}
    Since $\Gamma$ is almost planar, there exists, $\kappa \geq 1$, a connected, bounded-degree, planar graph $\Pi$, and $\kappa$-regular map $f : \Gamma \to \Pi$. By passing to a subgraph of $\Pi$, we can assume that $f$ is coarsely surjective.
    By Lemmas~\ref{lem:one-ended-regular-maps},~\ref{lem:small-cuts-regular-maps}, we may assume without loss of generality that $\Pi$ is 2-connected\footnote{Recall that a graph is \emph{2-connected} if it is connected and contains no cut-vertices.} and one-ended. Since $\Pi$ is one-ended, we may assume it comes equipped with a fixed topological embedding $\vartheta : \Pi \into \R^2$ where the image of $V(\Pi)$ in the plane is closed and discrete. This follows from the fact that the end-compactification of $\Pi$ embeds into the 2-sphere {\cite[Lem.~12]{richter20023}}.

    We now embed $\Pi$ as a subgraph of the 1-skeleton of a simplicial complex of bounded geometry which homeomorphic to the plane, in the following way.
    For every $n > 0$, let $B_n$ denote a simplicial complex homeomorphic to the disk, such that the boundary $\partial D_n$ contains exactly $n$ edges. We choose these $B_n$ such that the following two conditions hold:
    \begin{enumerate}
        \item\label{itm:uniformly-bounded-degree} The $B_n$ have uniformly bounded degree.

		\item\label{itm:uniform-filling-radii} There exists $k > 0$ such that for all $n > 0$ and all simplicial simple closed curves $C \subset B_n$, if $D \subset B_n$ is the disk bound by $C$, we have that $D$ is contained in the $R$-neighbourhood of $C$, where $R = k \log(\length(C))$. 
    \end{enumerate}

    Item~(\ref{itm:uniform-filling-radii}) can be rephrased as saying that simple closed curves in the $B_n$ have \emph{uniform logarithmic filling radii}.
    This can be achieved by taking $B_n$ to be uniformly quasi-isometric to a disk in the hyperbolic plane of circumference equal to $n$. Indeed, it is a standard fact that (\ref{itm:uniform-filling-radii}) holds in the hyperbolic plane, which follows from the hyperbolic isoperimetric inequality. 
    We also fix $H$ to be a connected simplicial complex of bounded degree, which is homeomorphic to the half-plane, and quasi-isometric to the \emph{hyperbolic half-plane}. That is, a connected component of $\bbH^2 \setminus \gamma$, where $\gamma$ is some bi-infinite geodesic. We have that $H$ also satisfies items~(\ref{itm:uniformly-bounded-degree}) and (\ref{itm:uniform-filling-radii}) above. Assume inconsequentially that it does so with the same constants.
    
    Given a facial circuit $C$ of $\Pi$ of length $n$, we glue a copy of $B_n$ to $C$ along its boundary.
    If there is a bi-infinite facial line $L$ in $\Pi$, we glue a copy of $H$ to $L$ along its boundary. The result is a connected, bounded-degree, simplicial 2-complex $P$ such that $P$ is homeomorphic to the plane, and $\Pi$ includes into the 1-skeleton of $P$ as a subgraph. Let $\iota : \Pi \into P^{(1)}$ denote this inclusion. Let $d > 0$ denote the maximum degree of the any vertex in $P^{(1)}$. 

    Suppose now that $P$ is not Gromov-hyperbolic. We will show that $\Gamma$ has at most quadratic growth by applying the Varopoulos inequality.
    If $P$ is not Gromov-hyperbolic, then there exists a sequence 
        $
        C_1, C_2, \ldots
        $
        of 18-bi-Lipschitz embedded simplicial cycles in $P$ of increasing length \cite[Prop.~5.1]{hume2020poorly}, say 
        $
        \length(C_n) =: \ell_n \to \infty
        $ 
        as $n \to \infty$. 
        By the Jordan curve theorem, we have that $C_n$ bounds a (simplicial) disk $D_n \subset P$.

        \begin{claim}\label{claim:linear-filling-diameter}
            There exists $z_n \in D_n$ such that $\dist_P(z_n, C_n) \geq \ell_n /200$. 
        \end{claim}

        \begin{proof}
            Suppose not, so $D_n \subset B_P(C_n;\ell_n/200)$. Consider $D_n$ with its own intrinsic path metric. Parametrise $C_n = \partial D_n \cong S^1$ at unit speed, say $C_n : [0,\ell_n] \to D_n$. For $k=0,1,2,3$ let 
            $$
            p_{n}^k := C_n([k\ell_n/4, (k+1)\ell_n/4]),
            $$
            Thus subdividing $C_n$ into four segments of equal length.
            Let $R_n^k \subset D_n$ be defined as
            $$
            R_n^k = \{x \in D_n : \dist_{P}(x, p_n^k) \leq \ell_n/200 \}. 
            $$
            It is clear that the $R_n^k$ are closed and connected, and contain $p_n^k$. 
            By our hypothesis, we have that 
            $$
            D_n = \bigcup _{k=1}^4 R_n^k.
            $$
            By the hex theorem, we must have that either 
            $$
            R_n^1 \cap R_n^3 \neq \emptyset  \ \ \ \  \text{or} \ \ \ \ R_n^2 \cap R_n^4 \neq \emptyset.
            $$
            Indeed, we may model the disk $D_n$ with arbitrary precision as a hex board, colouring a cell `white' if it intersects  in $R_n^1 \cup R_n^3$, and `black' if it intersects  in $R_n^2 \cup R_n^4$, picking arbitrarily to break ties. The hex theorem then easily implies that either $\dist_P(R_n^1, R_n^3) < \varepsilon$ or $\dist_P(R_n^2,R_n^4) < \varepsilon$ for all $\varepsilon > 0$. Since the $R_n^k$ are closed, this implies the claim. 
            
            Without loss of generality, suppose there exists $z \in R_n^1 \cap R_n^3$. Then there exists $a \in p_n^1$, $b \in p_n^3$ such that $\dist_P(a,b) \leq \ell_n/100$. 
            However, we know that $\dist_{C_n}(a,b) \geq \ell_n/4$. 
            In particular, this means that 
            $$
            \dist_{C_n}(a,b) \geq \ell_n/4 >  18\ell_n/100 \geq 18 \dist_P(a,b).
            $$
            This contradicts the fact that the inclusion $C_n \into P$ is 18-bi-Lipschitz. 
        \end{proof}

        \begin{claim}
            There exists a sequence $(x_n)$ of points in $\Gamma$ such that 
            $$
            \iota \circ f(x_n) \in D_n, \ \ \ \text{and} \ \ \ \frac{\dist_P(\iota \circ f(x_n), C_n)}{\ell_n} \not\to 0. 
            $$
        \end{claim}

        \begin{proof}
            Suppose this were not the case. So for all $n > 0$ we have that 
            $$
            \iota \circ f(\Gamma) \cap D_n = \iota(\Pi) \cap D_n \subset B_P(C_n;r_n),
            $$
            where $(r_n)$ is some increasing sequence such that $r_n /\ell_n \to 0$. 
            Consider the sequence of points $z_n \in D_n$ given by Claim~\ref{claim:linear-filling-diameter}. By passing to a subsequence, we may assume without loss of generality that all $z_n$ lie within $P \setminus \iota(\Pi)$, and so lie within one of the disks or half-planes which we glued to $\Pi$ to form $P$. Let $Q_n \subset P$ denote this piece.

            Consider the region $D_n \cap Q_n$. Note that $\partial (D_n \cap Q_n)$ is contained in the $r_n$-neighbourhood of $C_n$. 
            Let $C'_n \subset \partial (D_n \cap Q_n)$ denote a simple closed curve which separates $z_n$ from infinity, which must exist since $D_n$ is finite. 
            By the bounded geometry of $P$, we have that
            $$
            \length(C_n') \leq \ell_n \cdot d^{r_n}.
            $$
            Let $D_n'$ denote the disk bounded by $C_n'$. Since $C_n'$ is contained in $Q_n$, we have that $D_n'$ is also contained in $Q_n$. We also have that $z_n \in D_n'$. But now, by property~(\ref{itm:uniform-filling-radii}) of the $Q_n$, we have that any curve inside of $Q_n'$ has filling radius which is logarithmic in its length. In other words, we have that 
            \begin{align*}
                \dist_P(z_n, C_n) &\leq \dist_P(z_n, C_n') + r_n \\
                &< k \log (\length(C_n')) + r_n\\
                &\leq k \log (\ell_n \cdot d^{r_n}) + r_n\\
                & \leq k \log (\ell_n) + k\log(d)r_n + r_n.
            \end{align*}
            Since $r_n/\ell_n \to 0$, this bound eventually becomes smaller than $\ell_n/200$ for $n$ sufficiently large. 
            But this is a contradiction to our choice of $z_n$, and so we are done.
        \end{proof}

        We are now ready to show that $\Gamma$ has quadratic growth. 

        \begin{claim}
            The graph $\Gamma$ has at most quadratic growth. 
        \end{claim}

        \begin{proof}
            Let $A_n = (\iota \circ f)^{-1}(C_n)$. We see that $A_n$ separates $x_n$ from infinity, and $\dist_\Gamma(x_n, A_n)$ is at least linear in $\ell_n$. Moreover, since $f$ is regular we have that $|A_n| \leq \kappa |C_n| = \kappa \ell_n$. This implies that $\Gamma$ has at most quadratic growth by Corollary~\ref{cor:varopoulos}. 
        \end{proof}

        Thus, it has been shown that either $P$ is Gromov-hyperbolic, or the graph $\Gamma$ has quadratic growth. In all that follows, let us assume that $\Gamma$ does not have quadratic growth, and so the plane $P$ is indeed Gromov-hyperbolic. Fix $\delta > 0$ such that $P$ is $\delta$-hyperbolic. 

        \begin{claim}
            The plane $P$ is visual. 
        \end{claim}

        \begin{proof}
            We will prove this by applying Lemma~\ref{prop:criterion-for-visibility}. 
            Suppose there exists a sequence $C_n$ of simple closed combinatorial loops in $P$, all of length at most $100\delta$, such that the number of vertices in the disk $D_n$ bound by $C_n$ tends to infinity. Since $P$ has bounded geometry, there exists a sequence $(p_n)$ such that $p_n \in D_n$ where $\dist_P(p_n, C_n) \to \infty$. 
            
            By Proposition~\ref{lem:finite-components-in-fg-groups}, it is clear that there exists $m > 0$ such that for all $n > 0$ and all $x \in \Gamma$, if $\iota \circ f(x) \in D_n$, then 
            $$
            \dist_P(\iota\circ f(x), C_n) < m.
            $$
			This means that each $p_n$ lies in one of the hyperbolic components of $P \setminus \iota(\Pi)$, say $p_n \in Q_n$ where $Q_n$ is either a hyperbolic disk or a hyperbolic half-plane. Consider $Q_n \cap D_n$. We have that $\partial (Q_n \cap D_n)$ contains a simple closed curve $C_n'$ which separates $p_n$ from infinity, which is contained in a uniform neighbourhood of $C_n$. So then the $C_n'$ are uniformly bounded in length, but the disks they bound (which are necessarily contained inside of $Q_n$) grow arbitrarily large in filling radius. This violates property~(\ref{itm:uniform-filling-radii}) above. 

            Therefore, any simple closed curve in $P$ of length at most $100\delta$ bounds a disk of uniformly bounded area. By Lemma~\ref{prop:criterion-for-visibility}, this implies that $P$ is visual. 
        \end{proof}

        This concludes the proof of the proposition.
\end{proof}

We are now ready to conclude our proof of Theorem~\ref{thm:intro-hyperbolic-plane}, which we restate for convenience.

\mapintohypplane*

\begin{proof}
	By replacing $\Gamma$ by a quasi-isometric graph, we may assume that $\Gamma$ is vertex-transitive \cite[Prop.~5.1]{woess1994topological}. By Proposition~\ref{prop:map-into-gromov-hyp-plane}, we have that either $\Gamma$ has growth which is at most quadratic, or $\Gamma$ admits a regular map into the 1-skeleton of a simplicial complex which is Gromov-hyperbolic, visual, and homeomorphic to the plane $P$. 
    If the former is true, then $\Gamma$ is necessarily quasi-isometric to the Euclidean plane by \cite{woess1994topological}. If the latter is true, then by Theorem~\ref{thm:planes-QI-to-hypplane}, $P$ is quasi-isometric to $\mathbb H^2$. Since quasi-isometries are examples of regular maps, we are done. 
\end{proof}

\section{Proof of Theorem~\ref{thm:main-result}}

We now march towards a proof of our main result, Theorem~\ref{thm:main-result}. Currently, we have access to regular map $f : \Gamma \to \mathbb H^2$ from our graph $\Gamma$ to into the hyperbolic plane (assuming $\Gamma$ is not quasi-isometric to the Euclidean plane). This map still has the potential to be quite poorly behaved. The goal of this section is to take such a map and replace it with something `nicer'; ultimately `nice' will mean a quasi-isometry. We will do this in a three steps, ultimately proving the following theorem.

\begin{theorem}\label{thm:upgrading-reg-to-qi}
    Let $\Gamma$ be a connected, locally finite,  coarsely simply connected, quasi-transitive graph of asymptotic dimension at least 2.
    If $\Gamma$ admits a regular map into $\mathbb{H}^2$ then $\Gamma$ is quasi-isometric to $\mathbb{H}^2$.
\end{theorem}

Note that if a connected, locally finite,  coarsely simply connected, quasi-transitive graph $\Gamma$ is one-ended, then we automatically have that $\operatorname{asdim}(\Gamma) \geq 2$ \cite{fujiwara2007note}.
For the results in this section we will make heavy use of the fact that $\Gamma$ is coarsely simply connected, and essentially work entirely with a simply connected filling $X$ of $\Gamma$ (see Definition~\ref{def:csc}).
To facilitate this, the following easy lemma will be useful.

\begin{lemma}\label{lem:cts-extension}
    Let $X$ be connected polygonal 2-complex of bounded geometry, with 2-cell boundaries of uniformly bounded length.  Suppose there exists a regular map $f : X^{(1)} \to \mathbb H^2$. Then there exists a continuous regular map $\widehat f : X \to \mathbb H^2$. 
\end{lemma}

\begin{proof}
    Construct $\widehat f$ as follows. On vertices, $\widehat f$ agrees with $f$. That is, $\widehat f|_{X^{(0)}} = f|_{X^{(0)}}$. We extend $\widehat f$ to the 1-skeleton by sending edges to geodesics, and to the 2-skeleton by sending each 2-cell $\sigma$ to a null-homotopy of $\widehat f(\partial \sigma)$ of minimal area. It is an easy exercise to verify that $\widehat f$ is indeed continuous and regular.
\end{proof}

Throughout this section, we will fix $X$ to be a connected, cocompact polygonal 2-complex of bounded geometry and asymptotic dimension  at least 2, and  $f : X \to \mathbb H^2$ to be a continuous regular map. It will be convenient to replace $\mathbb H^2$ with some fixed vertex-transitive triangulation $P$ of $\mathbb H^2$, and also replace $X$ with some uniform barycentric subdivision, then assume without loss of generality that $f$ is simplicial. Up to modifying $f$ by a bounded amount, this assumption is harmless. This convention will follow us throughout the rest of this section.

\subsection{Width-visibility}

The first improvement we make relates to something we call \emph{width visibility}. Morally, this asks that if our map $f$ is at most $\kappa$-to-one everywhere, then anywhere we are in the space $X$, we should be able to `see' this full multiplicity in a uniform ball.

\begin{definition}
    The \emph{width} of a simplicial regular map $f : X \to P$ is defined to be 
    $$
    \width(f) = \max\Big\{\big|f^{-1}(v)\big|:v \in V(P)\Big\}.
    $$
\end{definition}

In other words, since $f$ is simplicial we have that the width of $f$ is the minimum $\kappa$ such that $f$ is $\kappa$-regular. 
The following property is of interest.

\begin{definition}[Width-visible]\label{defn:widthvisible}
    Let $R > 0$.
    We say that a simplicial regular map $f : X \to P$ is \emph{$R$-width-visible} if  for all $x \in V(X)$ there exists pairwise distinct vertices $y_1, \ldots, y_k \in B_X(x;R)$ such that
    $$
    f(y_1) = f(y_2) =\ldots = f(y_k),
    $$
    where $k = \width(f)$. 
\end{definition}

Note that we do \emph{not} ask that $f(y_i) = f(x)$.
If the constant $R$ is not important, we may just refer to such an $f$ as \emph{width-visible} and suppress $R$ from the notation.
We now prove the following. 

\begin{proposition}\label{prop:rank-visibility-whole-map}
    Suppose $f$ is such that $\width(f)$ is minimal amongst all simplicial regular maps $X \to P$. Then $f$ is width-visible.
\end{proposition}

\begin{proof}
    Let $f_{x,R} = f|_{B_X(x;R)}$. Then, slightly abusing notation, there must exist some $R > 0$ such that for all $x \in X$, we have that
    $$
    \width(f_{x,R}) = \width(f).
    $$
    If not, we may use the fact that $X$ and $P$ have cocompact automorphism groups and apply the Arzel\`a--Ascoli theorem to find another simplicial regular map $f' : X \to P$ satisfying $\width(f') < \width(f)$, which contradicts the minimality of $f$. The proposition follows.
\end{proof}

The following terminology will be helpful.

\begin{definition}[Special vertices]\label{def:special}
    Suppose $f : X \to P$ is $R$-width-visible.
	A vertex $u \in V(P)$ is said to be \emph{special} if $u \in f(X)$ and 
    $$
    |f^{-1}(u)|=\width(f) \ \ \ \text{and} \ \ \ \diam(|f^{-1}(u)|) \leq 2R.
    $$ 
    The set of special vertices is denoted $\mathscr S$. 
\end{definition}

\subsection{Stability}

We now move onto the second property on our manifesto. This will pertain to our map being `stable' in a particular sense. This will be a bit more technical. The ideas in this section take inspiration from the methods of Bonk and Kleiner used in \cite{bonk2002rigidity}.

In order to give the precise definition for what we mean when we say that $f$ is `stable', we will need to introduce the notion of a `local modification'.

\begin{definition}[Local modification]
    Let $\varepsilon > 0$, $x_0 \in X$. Then an \emph{$\varepsilon$-local modification of $f$ at $x_0$} is a continuous (not necessarily simplicial) map $f' : X \to P$ such that  $f|_U = f'|_U$, where $U = X \setminus B_X(x_0,\varepsilon)$.

\end{definition}

Intuitively, the first condition just says that $f'$ agrees with $f$ everywhere except for on a controlled ball. 
We now define stability as follows.

\begin{definition}[Stability]\label{defn:stability}
    Let $\varepsilon > 0$, $x_0 \in X$. We say that $f$ is \emph{$\varepsilon$-stable at $x_0$} if for all $\varepsilon$-local modifications $f'$ of $f$ at $x_0$, we have that $f(x_0) \in f'(X)$.
    If $f$ is $\varepsilon$-stable at $x_0$ for all $\varepsilon > 0$, then $f$ is said to be \emph{very stable} at $x_0$. 
\end{definition}

The following lemma is straightforward, and follows from the fact that $f$ is a proper map and $X$ and $P$ are locally compact.

\begin{lemma}\label{lem:local-mod-regular}
    Let $f'$ be a local modification of $f$. Then $f'$ satisfies
    $$
    \sup_{x \in X} \dist_P(f(x), f'(x)) < \infty.
    $$ 
    In particular, $f'$ is a regular map.
\end{lemma}


While we know by Lemma~\ref{lem:local-mod-regular} that local modifications are still regular maps, we do not quite yet have the uniformity required to be able to compose local modifications on a sufficiently sparse set and still preserve regularity. For this, we will use the following quantitative elaboration.

\begin{lemma}\label{lem:bounded-local-mod}
    Let $f'$ be a $\varepsilon$-local modification of $f$ at $x_0$, such that $f(x_0) \not\in f'(X)$. Then there exists another $\varepsilon$-local modification $f''$ of $f$ at $x_0$ such that  $f(x_0) \not\in f''(X)$, and also 
    $$
    \sup_{x \in X} \dist_P(f(x), f''(x)) \leq 2\varepsilon.
    $$ 
\end{lemma}

\begin{proof}
    Let $C$ be a circle of radius $\varepsilon$ about $f(x_0)$. Let $D$ denote the disk bounded by $C$. Let $S$ denote the sphere of radius $\varepsilon$ about $x_0$ in $X$, and $B$ denote the ball of the same radius about $x_0$ bounded by $S$.
    Note that since $f$ is 1-Lipschitz, we have that $f(B) \subset D$, and also $f'(S) \subset D$. However, $f'(B)$ could go wildly beyond $C$ by some large (but finite) amount. Consider a large circular annulus $A$ in $P$ with inner circle $C$ and outer circle $C'$, where $f'(B) \setminus D \subset A$. Let $D'$ denote the disk bounded by $C'$. Let $r : D' \to D$ be the geodesic retraction which maps $A$ onto $C$, and acts as the identity on $D$. We now define $f''$ piecewise via
    $$
        f''(x) =
        \begin{cases}
        r \circ f'(x) &\text{if $x \in B$,}\\
        f'(x) = f(x)  &\text{if $x \not\in B$.}\\
        \end{cases}
    $$
    It is not hard to see that $f''$ is a local modification satisfying our requirements. See Figure~\ref{fig:retracting-annulus} for a cartoon of this construction.
\end{proof}

\begin{figure}
    \centering
    \tikzset{every picture/.style={line width=0.75pt}} 

\begin{tikzpicture}[x=0.75pt,y=0.75pt,yscale=-1,xscale=1]

\draw   (376.75,160.17) .. controls (376.75,132.87) and (398.87,110.75) .. (426.17,110.75) .. controls (453.46,110.75) and (475.58,132.87) .. (475.58,160.17) .. controls (475.58,187.46) and (453.46,209.58) .. (426.17,209.58) .. controls (398.87,209.58) and (376.75,187.46) .. (376.75,160.17) -- cycle ;
\draw  [color={rgb, 255:red, 208; green, 2; blue, 27 }  ,draw opacity=1 ][fill={rgb, 255:red, 247; green, 217; blue, 217 }  ,fill opacity=1 ][dash pattern={on 3pt off 0.75pt}] (409.33,113.75) .. controls (416.58,111.5) and (419.08,111.25) .. (422.58,111.25) .. controls (426.08,111.25) and (424.31,118.01) .. (428.08,119.75) .. controls (431.86,121.49) and (461.58,126.66) .. (461.58,126.25) .. controls (461.58,125.84) and (461.33,125.5) .. (461.58,126.25) .. controls (461.83,127) and (463.83,127.75) .. (469.08,135.75) .. controls (474.33,143.75) and (476.33,159.75) .. (475.08,168.25) .. controls (473.83,176.75) and (465.58,190.25) .. (464.58,190.75) .. controls (463.58,191.25) and (454.34,194.9) .. (449.97,197.83) .. controls (445.6,200.76) and (441.33,204.75) .. (439.08,207.25) .. controls (436.83,209.75) and (419.08,212.25) .. (404.08,204) .. controls (389.08,195.75) and (381.58,182.5) .. (380.58,180.25) .. controls (379.58,178) and (396.64,153.45) .. (398.83,142.5) .. controls (401.02,131.55) and (397.08,120.25) .. (398.58,119.25) .. controls (400.08,118.25) and (402.08,116) .. (409.33,113.75) -- cycle ;
\draw  [color={rgb, 255:red, 208; green, 2; blue, 27 }  ,draw opacity=1 ][fill={rgb, 255:red, 255; green, 255; blue, 255 }  ,fill opacity=1 ][dash pattern={on 3.75pt off 0.75pt}] (418.42,154.25) .. controls (424.58,150) and (427.17,154.25) .. (439.17,159.75) .. controls (451.17,165.25) and (440.42,169.5) .. (434.67,170.25) .. controls (428.92,171) and (416.67,170.5) .. (417.67,165.25) .. controls (418.67,160) and (412.25,158.5) .. (418.42,154.25) -- cycle ;
\draw  [color={rgb, 255:red, 208; green, 2; blue, 27 }  ,draw opacity=1 ][fill={rgb, 255:red, 247; green, 217; blue, 217 }  ,fill opacity=1 ][dash pattern={on 3pt off 0.75pt}] (192.08,91.25) .. controls (192.33,74) and (210.23,78.75) .. (208.52,96.52) .. controls (206.8,114.28) and (209.41,130.48) .. (256.42,127) .. controls (303.43,123.52) and (266.39,126.28) .. (271.06,125.97) .. controls (275.74,125.66) and (311.7,117.43) .. (291.83,151.25) .. controls (271.97,185.07) and (302.41,187.99) .. (277.05,190.96) .. controls (251.69,193.93) and (251.72,193.63) .. (248.42,196.75) .. controls (245.12,199.87) and (230.16,213.61) .. (212.15,228.58) .. controls (194.13,243.54) and (175.51,245.53) .. (180.58,216.25) .. controls (185.65,186.97) and (158.37,213.59) .. (172.54,192.47) .. controls (186.71,171.35) and (195.34,157.55) .. (197.17,140.5) .. controls (199,123.45) and (191.68,116.21) .. (183.73,104.49) .. controls (175.78,92.76) and (191.83,108.5) .. (192.08,91.25) -- cycle ;
\draw  [color={rgb, 255:red, 208; green, 2; blue, 27 }  ,draw opacity=1 ][fill={rgb, 255:red, 255; green, 255; blue, 255 }  ,fill opacity=1 ][dash pattern={on 3.75pt off 0.75pt}] (215.17,155.5) .. controls (221.33,151.25) and (223.92,155.5) .. (235.92,161) .. controls (247.92,166.5) and (237.17,170.75) .. (231.42,171.5) .. controls (225.67,172.25) and (213.42,171.75) .. (214.42,166.5) .. controls (215.42,161.25) and (209,159.75) .. (215.17,155.5) -- cycle ;
\draw   (174.25,161.17) .. controls (174.25,133.87) and (196.37,111.75) .. (223.67,111.75) .. controls (250.96,111.75) and (273.08,133.87) .. (273.08,161.17) .. controls (273.08,188.46) and (250.96,210.58) .. (223.67,210.58) .. controls (196.37,210.58) and (174.25,188.46) .. (174.25,161.17) -- cycle ;
\draw  [dash pattern={on 4.5pt off 4.5pt}] (135.57,161.17) .. controls (135.57,112.51) and (175.01,73.07) .. (223.67,73.07) .. controls (272.32,73.07) and (311.76,112.51) .. (311.76,161.17) .. controls (311.76,209.82) and (272.32,249.26) .. (223.67,249.26) .. controls (175.01,249.26) and (135.57,209.82) .. (135.57,161.17) -- cycle ;
\draw  [color={rgb, 255:red, 208; green, 2; blue, 27 }  ,draw opacity=1 ][line width=1.5]  (216.92,139.75) .. controls (235.42,176) and (255.67,131.25) .. (252.92,146.25) .. controls (250.17,161.25) and (275.42,161.25) .. (256.92,178.75) .. controls (238.42,196.25) and (167.42,194) .. (190.42,183.75) .. controls (213.42,173.5) and (198.42,103.5) .. (216.92,139.75) -- cycle ;
\draw [color={rgb, 255:red, 0; green, 0; blue, 0 }  ,draw opacity=1 ]   (223.67,161.17) ;
\draw [shift={(223.67,161.17)}, rotate = 0] [color={rgb, 255:red, 0; green, 0; blue, 0 }  ,draw opacity=1 ][fill={rgb, 255:red, 0; green, 0; blue, 0 }  ,fill opacity=1 ][line width=0.75]      (0, 0) circle [x radius= 3.35, y radius= 3.35]   ;
\draw  [color={rgb, 255:red, 208; green, 2; blue, 27 }  ,draw opacity=1 ][line width=1.5]  (419.42,138.75) .. controls (437.92,175) and (458.17,130.25) .. (455.42,145.25) .. controls (452.67,160.25) and (477.92,160.25) .. (459.42,177.75) .. controls (440.92,195.25) and (369.92,193) .. (392.92,182.75) .. controls (415.92,172.5) and (400.92,102.5) .. (419.42,138.75) -- cycle ;
\draw [color={rgb, 255:red, 0; green, 0; blue, 0 }  ,draw opacity=1 ]   (426.17,160.17) ;
\draw [shift={(426.17,160.17)}, rotate = 0] [color={rgb, 255:red, 0; green, 0; blue, 0 }  ,draw opacity=1 ][fill={rgb, 255:red, 0; green, 0; blue, 0 }  ,fill opacity=1 ][line width=0.75]      (0, 0) circle [x radius= 3.35, y radius= 3.35]   ;
\draw [color={rgb, 255:red, 155; green, 155; blue, 155 }  ,draw opacity=1 ]   (252.08,85.5) -- (243.57,107.14) ;
\draw [shift={(242.83,109)}, rotate = 291.49] [color={rgb, 255:red, 155; green, 155; blue, 155 }  ,draw opacity=1 ][line width=0.75]    (6.56,-1.97) .. controls (4.17,-0.84) and (1.99,-0.18) .. (0,0) .. controls (1.99,0.18) and (4.17,0.84) .. (6.56,1.97)   ;
\draw [color={rgb, 255:red, 155; green, 155; blue, 155 }  ,draw opacity=1 ]   (249.67,237.25) -- (243.94,219.16) ;
\draw [shift={(243.33,217.25)}, rotate = 72.43] [color={rgb, 255:red, 155; green, 155; blue, 155 }  ,draw opacity=1 ][line width=0.75]    (6.56,-1.97) .. controls (4.17,-0.84) and (1.99,-0.18) .. (0,0) .. controls (1.99,0.18) and (4.17,0.84) .. (6.56,1.97)   ;
\draw [color={rgb, 255:red, 155; green, 155; blue, 155 }  ,draw opacity=1 ]   (292.58,200.75) -- (280.87,194.89) ;
\draw [shift={(279.08,194)}, rotate = 26.57] [color={rgb, 255:red, 155; green, 155; blue, 155 }  ,draw opacity=1 ][line width=0.75]    (6.56,-1.97) .. controls (4.17,-0.84) and (1.99,-0.18) .. (0,0) .. controls (1.99,0.18) and (4.17,0.84) .. (6.56,1.97)   ;
\draw [color={rgb, 255:red, 155; green, 155; blue, 155 }  ,draw opacity=1 ]   (145.92,161.5) -- (166.08,161.5) ;
\draw [shift={(168.08,161.5)}, rotate = 180] [color={rgb, 255:red, 155; green, 155; blue, 155 }  ,draw opacity=1 ][line width=0.75]    (6.56,-1.97) .. controls (4.17,-0.84) and (1.99,-0.18) .. (0,0) .. controls (1.99,0.18) and (4.17,0.84) .. (6.56,1.97)   ;

\draw (225.67,164.57) node [anchor=north west][inner sep=0.75pt]  [font=\scriptsize,color={rgb, 255:red, 0; green, 0; blue, 0 }  ,opacity=1 ]  {$f( x_{0})$};
\draw (197.42,177.82) node [anchor=north west][inner sep=0.75pt]  [font=\scriptsize,color={rgb, 255:red, 208; green, 2; blue, 27 }  ,opacity=1 ]  {$f( S)$};
\draw (185.42,216.32) node [anchor=north west][inner sep=0.75pt]  [font=\scriptsize,color={rgb, 255:red, 208; green, 2; blue, 27 }  ,opacity=1 ]  {$f'( B)$};
\draw (478.17,138.57) node [anchor=north west][inner sep=0.75pt]  [font=\scriptsize,color={rgb, 255:red, 208; green, 2; blue, 27 }  ,opacity=1 ]  {$f''( B)$};
\draw (326,150.65) node [anchor=north west][inner sep=0.75pt]    {$\Longrightarrow $};
\draw (152,141.65) node [anchor=north west][inner sep=0.75pt]  [color={rgb, 255:red, 155; green, 155; blue, 155 }  ,opacity=1 ]  {$r$};
\draw (286,80.9) node [anchor=north west][inner sep=0.75pt]    {$C'$};
\draw (253.25,202.15) node [anchor=north west][inner sep=0.75pt]    {$C$};

\end{tikzpicture}
    \caption{Proof of Lemma~\ref{lem:bounded-local-mod}.}
    \label{fig:retracting-annulus}
\end{figure}
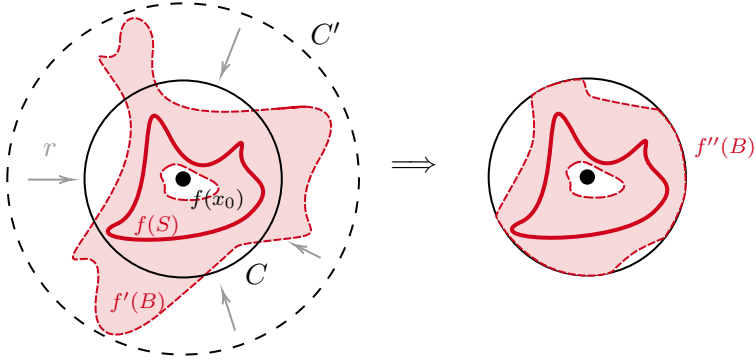

The following fact is nice to note.

\begin{proposition}\label{prop:stable-surj}
    Suppose $f$ is very stable at $x_0$. Then $f$ is surjective. Moreover, given $x \in X$, if $f'$ is a local modification of $f$ at $x$ then $f'$ is also surjective.
\end{proposition}

\begin{proof}
    Suppose $f$ not surjective. Write $y_0 = f(x_0)$. Fix some $z \in P \setminus f(X)$. It is easy to see that there exists a diffeomorphism $F : P \to P$ such that $F(z) = y_0$, where $F$ acts as the identity outside of a fixed ball $B$ around $y_0$. We can then consider the map $f' := F \circ f$. Since $f$ is regular, we have that $f'$ is a $\delta$-local modification at $x_0$ for some large choice of $\delta > 0$. Note that $y_0 \not \in f'(X)$. In particular, this shows that $f$ is not $\delta$-stable at $x_0$.

    For the second claim, suppose $f''$ is a local modification at $x$ which is not surjective. Then we may apply an identical argument to show that $f$ is not $\delta$-stable at $x_0$ for some large $\delta$. 
\end{proof}

The style of argument in the proof of Proposition~\ref{prop:stable-surj} will be used several times in what is to come, so we encourage the reader to make sure they are comfortable with it.
Observe that this proposition has the following nice consequence.

\begin{proposition}\label{prop:verystableeverywhere}
    Suppose $f$ is very stable at some $x_0 \in X$. Then $f$ is very stable at every $x \in X$.
\end{proposition}

\begin{proof}
    If $f$ were not very stable at some $x \in X$, then there is a local modification of $f$ which is not surjective. But then $f$ cannot be very stable by Proposition~\ref{prop:stable-surj}.
\end{proof}

In light of Proposition~\ref{prop:verystableeverywhere}, it makes sense to refer $f$ as simply being \emph{very stable}, with no reference to a particular basepoint.
We now proceed with proving the following theorem, which is the main result of this subsection.

\begin{theorem}\label{thm:stable-map-exists}
    There exists a regular simplicial map $f : X \to P$ which is width-visible and very stable.
\end{theorem}

\begin{proof}
    Fix a regular simplicial map $f : X \to P$ such that $\width(f)$ is minimal amongst all such maps. By Proposition~\ref{prop:rank-visibility-whole-map} we have that that $f$ is $R$-width-visible for some $R > 0$.  We first have the following claim.

    \begin{claim}
        For all $\varepsilon > 0$ there exists a vertex $x \in V(X)$ such that $f$ is $\varepsilon$-stable at $x$. Moreover, we can take each $x$ such that $f(x)$ is special.
    \end{claim}

    \begin{proof}
		Suppose this were not the case for some fixed $\varepsilon > 0$. Choose a maximal $10\varepsilon$-separated subset $S \subset P$. Since $f$ is $1$-Lipschitz, every $y \in f(X)$ is within distance $R$ of a special vertex, so by assuming $\varepsilon > R$ as we may, we can find such $S$ satisfying $S \cap f(X) \subset \mathscr S$.
        
        Using Lemma~\ref{lem:bounded-local-mod} and pasting together all of the local modifications at each $s \in S$ we produce a continuous regular map $f' : X \to P$  such that $f'(X) \cap S = \emptyset$.  In fact, it is easy to see that we may take $f'$ such that $f'(X) \cap B = \emptyset$, where $B = B_P(S;\varepsilon/5)$. However, since $f'$ is continuous and $X$ is simply connected, we may lift $f'$ to a regular map $\widetilde {f'} : X \to \widetilde {P \setminus B}$ to the universal cover. But $\widetilde {P \setminus B}$, which is a geodesic space, is easily seen to be quasi-isometric to a tree by Manning's bottleneck criterion \cite[Thm.~4.6]{manning2005geometry}. However, then $X$ has asymptotic dimension at most 1, which is a contradiction.
    \end{proof}

	Since both $X$ and $P$ are quasi-transitive, we can now translate all these increasingly stable points to one point $x_0$, giving a sequence of maps $f_n$ such that $f_n$ is $n$-stable at $x_0$, and $f_n(x_0) = f_m(x_0)$ for all $n,m \geq 1$. Now, use the Arzel\`a--Ascoli theorem to take a limit $\widehat f$ of the $f_n$. This will satisfy $\width(\widehat f) = \width(f)$, and so since the width of $f$ was chosen to be minimal we again have by Proposition~\ref{prop:rank-visibility-whole-map} that $\widehat f$ is width-visible. 
    We now verify the stability of $\widehat f$.

    \begin{claim}
         $\widehat f$ is very stable.
    \end{claim}

    \begin{proof}
        Write $y_0 = \widehat f(x_0)$. 
        Suppose $\widehat f$ is not $\varepsilon$-stable at $x_0$, for some fixed choice of $\varepsilon > 0$. Without loss of generality, let us assume that $\varepsilon > 10R$. Let $f'$ be a $\varepsilon$-local modification of $\widehat f$ at $x_0$ such that $y_0 \not \in f'(X)$. 
        Choose $n > 0$ such that 
        $$
        f_n|_{B(x_0,\varepsilon)} = \widehat f|_{B(x_0,\varepsilon)}.
        $$
        We also assume that $n$ is chosen to be sufficiently large so that $f_n$ is $\varepsilon$-stable at $x_0$.
        The local modification of $\widehat f$ induces a local modification of $f_n$, say $f'_n$. This also satisfies
        $$
        f_n'|_{B(x_0,\varepsilon)} = \widehat f'|_{B(x_0,\varepsilon)}.
        $$
        We claim that $ y_0 \not \in f'_n(X)$.
		Indeed, suppose there exists $z \in X$ such that $f_n'(z) = y_0$. Either $z \in B_X(x_0,\varepsilon)$ or $z \not \in B_X(x_0,\varepsilon)$. Since $\varepsilon > 2R$ and $\widehat f(x_0)=f_n(x_0)$ is special for $f_n$, the latter cannot happen, so we must have that $z \in B_X(x_0,\varepsilon)$. But in this case, since $f_n'$ and $\widehat f'$ agree on the $\varepsilon$-ball around $x_0$, this also cannot happen. Thus, such a $z$ does not exist, so $ y_0 \not \in f'_n(X)$.
        However, this is a contradiction, as $f_n$ is $\varepsilon$-stable at $x_0$. Thus, it follows that $\widehat f$ is $\varepsilon$-stable at $x_0$, and so since $\varepsilon$ was arbitrary we can conclude that $\widehat f$ is very stable. 
    \end{proof}

    This concludes the proof of the theorem.
\end{proof}

\subsection{Upgrading to a quasi-isometry}\label{sec:upgrade-to-qi}

The third and final step is to show that our width-visible, very stable map $f : X \to P$ is actually secretly a quasi-isometry.

\begin{theorem}
    A regular simplicial map $f:X \to P$ which is width-visible and very stable is a quasi-isometry.
\end{theorem}
\begin{proof}
    Since $f$ is surjective (by Proposition~\ref{prop:stable-surj}), coarsely Lipschitz and $P$ is geodesic, it suffices to show that $f$ is uniformly proper.

    Since $f$ is width-visible, it suffices to check that for all $x,x' \in f^{-1}(\mathscr{S})$, we have that 
    $$
    \dist_P(f(x),f(x')) \geq \rho_-(\dist_X(x,x')),
    $$
    for some $\rho_-(t)\to \infty$ as $t \to \infty$.

    Given $x \in f^{-1}(\mathscr{S})$, and given $T > 0$,
    let $V$ be the connected component of $f^{-1}(B(f(x),T+2R))$ which contains $f^{-1}(f(x))$.
    By regularity, $\diam(V) \leq C(T)$ for some function $C(T)$ independent of $x$.
    Suppose $x' \in f^{-1}(\mathscr{S})$ satisfies $\dist_X(x,x') > C(T)+2R$ but $\dist_P(f(x),f(x')) < T$.  Then $f(x') \notin f|_{B(x,C(T))}$.
    
    Let $F:P \to P$ be a diffeomorphism which is the identity outside $B(f(x),T)$ and inside $B(f(x),T)$ moves $f(x')$ to $f(x)$.
    Then define a local modification $f':X \to P$ of $f$ by $f' = f$ on $X \setminus V$, and $f' = F \circ f$ on $V$.
    Then $f(x) \notin f'(X)$, which contradicts that $f(x)$ is a very stable point of $f$ (Proposition~\ref{prop:verystableeverywhere}).
    Finally, let $\rho_-(t) = \max\{T:C(T)+R \leq t\}$.
\end{proof}

This concludes the proof of Theorem~\ref{thm:upgrading-reg-to-qi}, stated at the start of this section.
We are now ready to finish the proof of our main result, Theorem~\ref{thm:main-result}. The only work left to do is to observe that accessibility allows us to extend beyond the realm of one-ended graphs.

\mainresult*

\begin{proof}
    Let $\Gamma$ be an almost planar, coarsely simply connected, connected, locally finite, quasi-transitive graph. By \cite{hamann2018accessibility}, we have that $\Gamma$ is an accessible graph, as such it can be factorised into a `tree-amalgam' of a  finite collection $\{H_1, \ldots, H_n\}$ of connected, locally finite, quasi-transitive graphs with at most one end (see \cite{hamann2025tree} for definitions).  It is easy to see that these graphs are also almost planar and coarsely simply connected. By Theorems~\ref{thm:intro-hyperbolic-plane} and \ref{thm:upgrading-reg-to-qi}, each of the $H_i$ is either finite or quasi-isometric to either the Euclidean or hyperbolic plane. By \cite{hamann2025tree}, we conclude that $\Gamma$ is quasi-isometric to a free product of free and surface groups, and in particular is quasi-isometric to a planar graph. 
\end{proof}

We will finish this section by providing a direct proof of Corollary~\ref{cor:intro-fpgroups} which does not use the machinery of \cite{macmanus2023accessibility}.

\maincorollary*

\begin{proof}
    Let $G$ be an almost planar, finitely presented group.
    By Dunwoody's accessibility theorem \cite{dunwoody1985accessibility}, $G$ splits as a graph of groups with finite edge groups and finitely presented vertex groups $H_v$ with at most one end. By Theorems~\ref{thm:intro-hyperbolic-plane} and \ref{thm:upgrading-reg-to-qi}, each infinite $H_v$ is quasi-isometric to either the Euclidean plane or the hyperbolic plane. This implies that the vertex groups are all finite or virtual surface groups. Note that virtual surface groups are residually finite and virtually torsion-free.
    It is a classical fact that a group which splits as a graph of groups with finite edge groups and residually finite, virtually torsion-free edge groups is itself residually finite and virtually torsion free. In particular, $G$ has a finite-index subgroup $H$ which splits as a free product of free and surface groups, and so $H$ admits a planar Cayley graph \cite[Thm.~3]{arzhantseva2004cayley}.
\end{proof}

\bibliography{references}

@article{georgakopoulos2010lamplighter,
  title={Lamplighter graphs do not admit harmonic functions of finite energy},
  author={Georgakopoulos, Agelos},
  journal={Proceedings of the American Mathematical Society},
  pages={3057--3061},
  year={2010},
  publisher={JSTOR}
}

@article{benjamini1996harmonic,
  title={Harmonic functions on planar and almost planar graphs and manifolds, via circle packings},
  author={Benjamini, Itai and Schramm, Oded},
  journal={Inventiones mathematicae},
  volume={126},
  number={3},
  pages={565--587},
  year={1996},
  publisher={Springer}
}

@article{bonk2000embeddings,
 AUTHOR = {Bonk, M. and Schramm, O.},
     TITLE = {Embeddings of {G}romov hyperbolic spaces},
   JOURNAL = {Geom. Funct. Anal.},
  FJOURNAL = {Geometric and Functional Analysis},
    VOLUME = {10},
      YEAR = {2000},
    NUMBER = {2},
     PAGES = {266--306},
      ISSN = {1016-443X,1420-8970},
   MRCLASS = {53C23 (54E40 57M07)},
}

@article{georgakopoulos2025graph,
  title={Graph minors and metric spaces},
  author={Georgakopoulos, Agelos and Papasoglu, Panos},
  journal={Combinatorica},
  volume={45},
  number={3},
  pages={33},
  year={2025},
  publisher={Springer}
}

@article{bonamy2023asymptotic,
  title={Asymptotic dimension of minor-closed families and {Assouad--Nagata} dimension of surfaces},
  author={Bonamy, Marthe and Bousquet, Nicolas and Esperet, Louis and Groenland, Carla and Liu, Chun-Hung and Pirot, Fran{\c{c}}ois and Scott, Alexander},
  journal={Journal of the European Mathematical Society},
  volume={26},
  number={10},
  pages={3739--3791},
  year={2023}
}

@article {spielman-teng-07-spectral-planar,
    AUTHOR = {Spielman, Daniel A. and Teng, Shang-Hua},
     TITLE = {Spectral partitioning works: planar graphs and finite element
              meshes},
   JOURNAL = {Linear Algebra Appl.},
  FJOURNAL = {Linear Algebra and its Applications},
    VOLUME = {421},
      YEAR = {2007},
    NUMBER = {2-3},
     PAGES = {284--305},
      ISSN = {0024-3795,1873-1856},
   MRCLASS = {15A18 (05C50 05C70 05C90)},
  MRNUMBER = {2294342},
MRREVIEWER = {Nicola\ Guglielmi},
       DOI = {10.1016/j.laa.2006.07.020},
       URL = {https://doi-org.bris.idm.oclc.org/10.1016/j.laa.2006.07.020},
}

@article{benjamini2012separation,
  title={On the separation profile of infinite graphs},
  author={Benjamini, Itai and Schramm, Oded and Tim{\'a}r, {\'A}d{\'a}m},
  journal={Groups, Geometry, and Dynamics},
  volume={6},
  number={4},
  pages={639--658},
  year={2012}
}

@article{lazarovich2025hyperbolic,
  title={Hyperbolic groups with logarithmic separation profile},
  author={Lazarovich, Nir and Le Coz, Corentin},
  journal={Algebraic \& Geometric Topology},
  volume={25},
  number={1},
  pages={39--54},
  year={2025},
  publisher={Mathematical Sciences Publishers}
}

@inproceedings{gournay2023separation,
  title={Separation profile, isoperimetry, growth and compression},
  author={Gournay, Antoine and Le Coz, Corentin},
  booktitle={Annales de l'Institut Fourier},
  volume={73},
  number={4},
  pages={1627--1675},
  year={2023}
}

@article{hume2020poincare,
  title={Poincar{\'e} profiles of groups and spaces},
  author={Hume, David and Mackay, John M. and Tessera, Romain},
  journal={Revista Matem{\'a}tica Iberoamericana},
  volume={36},
  number={6},
  pages={1835--1886},
  year={2020}
}

@article{hume2022poincare,
  title={Poincar{\'e} profiles of {Lie} groups and a coarse geometric dichotomy},
  author={Hume, David and Mackay, John M. and Tessera, Romain},
  journal={Geometric and Functional Analysis},
  volume={32},
  number={5},
  pages={1063--1133},
  year={2022},
  publisher={Springer}
}

@book{woess2000random,
  title={Random walks on infinite graphs and groups},
  author={Woess, Wolfgang},
  number={138},
  year={2000},
  publisher={Cambridge University Press}
}

@article{medolla1995extension,
  title={Extension of {Foster’s} averaging formula to infinite networks with moderate growth},
  author={Medolla, Guia and Soardi, Paolo M},
  journal={Mathematische Zeitschrift},
  volume={219},
  number={1},
  pages={171--185},
  year={1995},
  publisher={Springer}
}

@article{macmanus2023accessibility,
  title={Accessibility, planar graphs, and quasi-isometries},
  author={MacManus, Joseph Paul},
  journal={arXiv preprint arXiv:2310.15242},
  year={2023}
}

@article{saloff1995isoperimetric,
  title={Isoperimetric inequalities and decay of iterated kernels for almost-transitive {Markov} chains},
  author={Saloff-Coste, Laurent},
  journal={Combinatorics, Probability and Computing},
  volume={4},
  number={4},
  pages={419--442},
  year={1995},
  publisher={Cambridge University Press}
}

@article{hume2020poorly,
  title={Poorly connected groups},
  author={Hume, David and Mackay, John M.},
  journal={Proceedings of the American Mathematical Society},
  volume={148},
  number={11},
  pages={4653--4664},
  year={2020}
}

@article{hume-mackay-25-round-trees,
Author = {David Hume and John M. Mackay},
Title = {Connecting conformal dimension and {Poincaré} profiles},
Year = {2025},
journal={arXiv preprint arXiv:2511.10469},
}

@article {TukiaVaisala-80-qs,
    AUTHOR = {Tukia, P. and V{\"a}is{\"a}l{\"a}, J.},
     TITLE = {Quasisymmetric embeddings of metric spaces},
   JOURNAL = {Ann. Acad. Sci. Fenn. Ser. A I Math.},
  FJOURNAL = {Suomalaisen Tiedeakatemian Toimituksia. Sarja A. Annales
              Academiae Scientiarum Fennicae. Series A I. Mathematica},
    VOLUME = {5},
      YEAR = {1980},
    NUMBER = {1},
     PAGES = {97--114},
      ISSN = {0066-1953},
     CODEN = {AAFMAT},
   MRCLASS = {30C60 (54C99)},
}

@article{macmanus2024note,
  title={A note on quasi-transitive graphs quasi-isometric to planar {(Cayley)} graphs},
  author={MacManus, Joseph Paul},
  journal={arXiv preprint arXiv:2407.13375},
  year={2024}
}

@article{lipton1979separator,
  title={A separator theorem for planar graphs},
  author={Lipton, Richard J and Tarjan, Robert Endre},
  journal={SIAM Journal on Applied Mathematics},
  volume={36},
  number={2},
  pages={177--189},
  year={1979},
  publisher={SIAM}
}

@inproceedings{ancona2007positive,
  title={Positive harmonic functions and hyperbolicity},
  author={Ancona, Alano},
  booktitle={Potential Theory Surveys and Problems: Proceedings of a Conference held in Prague, July 19--24, 1987},
  pages={1--23},
  year={2007},
  organization={Springer}
}

@book{bridson2013metric,
  title={Metric spaces of non-positive curvature},
  author={Bridson, Martin R and Haefliger, Andr{\'e}},
  volume={319},
  year={2013},
  publisher={Springer Science \& Business Media}
}

@article{hamann2018accessibility,
  title={Accessibility in transitive graphs},
  author={Hamann, Matthias},
  journal={Combinatorica},
  volume={38},
  number={4},
  pages={847--859},
  year={2018},
  publisher={Springer}
}

@article{maschke1896representation,
  title={The representation of finite groups, especially of the rotation groups of the regular bodies of three-and four-dimensional space, by {Cayley's} color diagrams},
  author={Maschke, Heinrich},
  journal={American Journal of Mathematics},
  volume={18},
  number={2},
  pages={156--194},
  year={1896},
  publisher={JSTOR}
}

@article{bowditch2004planar,
  title={Planar groups and the {Seifert} conjecture},
  author={Bowditch, Brian H},
  journal={Journal f{\"u}r die reine und angewandte Mathematik},
  volume={576},
  pages={11--62},
  year={2004}
}

@article{dunwoody1985accessibility,
  title={The accessibility of finitely presented groups.},
  author={Dunwoody, Martin J},
  journal={Inventiones mathematicae},
  volume={81},
  pages={449--458},
  year={1985}
}

@article{richter20023,
  title={3-connected planar spaces uniquely embed in the sphere},
  author={Richter, R. Bruce and Thomassen, Carsten},
  journal={Transactions of the American Mathematical Society},
  volume={354},
  number={11},
  pages={4585--4595},
  year={2002}
}

@book{zieschang1970flachen,
  title={Fl{\"a}chen und ebene diskontinuierliche {Gruppen}},
  author={Zieschang, Heiner and Vogt, Elmar and Coldewey, H-D},
  year={1970},
  publisher={Springer}
}

@article{wilkie1966non,
  title={On {non-Euclidean} crystallographic groups},
  author={Wilkie, Hugh Campbell},
  journal={Mathematische Zeitschrift},
  volume={91},
  pages={87--102},
  year={1966},
  publisher={Springer}
}

@article{dunwoody1993inaccessible,
  title={An inaccessible group},
  author={Dunwoody, Martin J},
  journal={Geometric group theory},
  volume={1},
  pages={75--78},
  year={1993},
  publisher={Cambridge University Press Cambridge}
}

@article{droms2006infinite,
title = {Infinite-ended groups with planar {Cayley} graphs},
author = {Carl Droms},
pages = {487--496},
volume = {9},
number = {4},
journal = {Journal of Group Theory},
year = {2006},
}

@article{georgakopoulos2019planari,
  title={The planar {Cayley} graphs are effectively enumerable {I}: consistently planar graphs},
  author={Georgakopoulos, Agelos and Hamann, Matthias},
  journal={Combinatorica},
  volume={39},
  number={5},
  pages={993--1019},
  year={2019},
  publisher={Springer}
}

@article{casson1994convergence,
  title={Convergence groups and {Seifert} fibered 3-manifolds},
  author={Casson, Andrew and Jungreis, Douglas},
  journal={Inventiones mathematicae},
  volume={118},
  number={1},
  pages={441--456},
  year={1994},
  publisher={Springer}
}

@article{gabai1992convergence,
  title={Convergence groups are {Fuchsian} groups},
  author={Gabai, David},
  journal={Annals of Mathematics},
  volume={136},
  number={3},
  pages={447--510},
  year={1992},
  publisher={JSTOR}
}

@article{mess1988seifert,
  title={The {Seifert} conjecture and groups which are coarse quasiisometric to planes},
  author={Mess, Geoffrey},
  journal={preprint},
  year={1988}
}

@article{tukia1988homeomorphic,
  title={Homeomorphic conjugates of {Fuchsian} groups.},
  author={Tukia, Pekka},
  journal={Journal f{\"u}r die reine und angewandte Mathematik},
  volume={390},
  pages={1--54},
  year={1988}
}

@article{maillot2001quasi,
  title={Quasi-isometries of groups, graphs and surfaces},
  author={Maillot, Sylvain},
  journal={Commentarii Mathematici Helvetici},
  volume={76},
  pages={29--60},
  year={2001},
  publisher={Springer}
}

@article{papasoglu2002quasi,
  title={Quasi-isometries between groups with infinitely many ends},
  author={Papasoglu, Panos and Whyte, Kevin},
  journal={Commentarii Mathematici Helvetici},
  volume={77},
  number={1},
  pages={133--144},
  year={2002},
  publisher={Springer}
}

@article{arzhantseva2004cayley,
  title={On the {Cayley} graph of a generic finitely presented group},
  author={Arzhantseva, Goulnara N and Cherix, P-A},
  journal={Bulletin of the Belgian Mathematical Society-Simon Stevin},
  volume={11},
  number={4},
  pages={589--601},
  year={2004},
  publisher={The Belgian Mathematical Society}
}

@article{davies2025string,
  title={String graphs are quasi-isometric to planar graphs},
  author={Davies, James},
  journal={arXiv preprint arXiv:2510.19602},
  year={2025}
}

@article{macmanus2025vertex,
  title={Vertex-transitive graphs with uniformly bisecting quasi-geodesics},
  author={MacManus, Joseph Paul},
  journal={arXiv preprint arXiv:2511.10759},
  year={2025}
}

@article{macmanus2025fat,
  title={Fat minors in finitely presented groups},
  author={MacManus, Joseph Paul},
  journal={Combinatorica},
  volume={45},
  number={4},
  pages={40},
  year={2025},
  publisher={Springer}
}

@article{woess1994topological,
  title={Topological groups and recurrence of quasi transitive graphs},
  author={Woess, Wolfgang},
  journal={Rendiconti del Seminario Matematico e Fisico di Milano},
  volume={64},
  number={1},
  pages={185--213},
  year={1994},
  publisher={Springer}
}

@article{bonk2002rigidity,
  title={Rigidity for quasi-{M}{\"o}bius group actions},
  author={Bonk, Mario and Kleiner, Bruce},
  journal={Journal of Differential Geometry},
  volume={61},
  number={1},
  pages={81--106},
  year={2002},
  publisher={Lehigh University}
}

@book {Hocking-Young-Topology,
    AUTHOR = {Hocking, John G. and Young, Gail S.},
     TITLE = {Topology},
   EDITION = {Second},
 PUBLISHER = {Dover Publications, Inc., New York},
      YEAR = {1988},
     PAGES = {x+374},
      ISBN = {0-486-65676-4},
   MRCLASS = {54-01 (55-01)},
  MRNUMBER = {1016814},
}

@article{hamann2025tree,
  title={Tree amalgamations and quasi-isometries},
  author={Hamann, Matthias},
  journal={Journal of Combinatorial Theory, Series B},
  volume={173},
  pages={83--101},
  year={2025},
  publisher={Elsevier}
}

@article{Esperet_Giocanti_2024, title={Coarse Geometry of Quasi-Transitive Graphs Beyond Planarity}, volume={31}, url={https://www.combinatorics.org/ojs/index.php/eljc/article/view/v31i2p41}, DOI={10.37236/12661}, abstractNote={&lt;p&gt;We study geometric and topological properties of infinite graphs that are quasi-isometric to a planar graph of bounded degree. We prove that every locally finite quasi-transitive graph excluding a minor is quasi-isometric to a planar graph of bounded degree. We use the result to give a simple proof of the result that finitely generated minor-excluded groups have Assouad-Nagata dimension at most 2 (this is known to hold in greater generality, but all known proofs use significantly deeper tools). We also prove that every locally finite quasi-transitive graph that is quasi-isometric to a planar graph is $k$-planar for some $k$ (i.e. it has a planar drawing with at most $k$ crossings per edge), and discuss a possible approach to prove the converse statement.&lt;/p&#38;gt;}, number={2}, journal={The Electronic Journal of Combinatorics}, author={Esperet, Louis and Giocanti, Ugo}, year={2024}, pages={P2.41} }

@article{georgakopoulos2023planar,
  title={{The planar Cayley graphs are effectively enumerable II}},
  author={Georgakopoulos, Agelos and Hamann, Matthias},
  journal={European Journal of Combinatorics},
  volume={110},
  pages={103668},
  year={2023},
  publisher={Elsevier}
}

@article{esperet2024structure,
  title={The structure of quasi-transitive graphs avoiding a minor with applications to the domino problem},
  author={Esperet, Louis and Giocanti, Ugo and Legrand-Duchesne, Cl{\'e}ment},
  journal={Journal of Combinatorial Theory, Series B},
  volume={169},
  pages={561--613},
  year={2024},
  publisher={Elsevier}
}

@article{grigoriev2007algorithms,
  title={Algorithms for graphs embeddable with few crossings per edge},
  author={Grigoriev, Alexander and Bodlaender, Hans L},
  journal={Algorithmica},
  volume={49},
  number={1},
  pages={1--11},
  year={2007},
  publisher={Springer}
}

@article{hopcroft1974efficient,
  title={Efficient planarity testing},
  author={Hopcroft, John and Tarjan, Robert},
  journal={Journal of the ACM (JACM)},
  volume={21},
  number={4},
  pages={549--568},
  year={1974},
  publisher={ACM New York, NY, USA}
}

@inproceedings{ringel1965sechsfarbenproblem,
  title={Ein {Sechsfarbenproblem} auf der kugel},
  author={Ringel, Gerhard},
  booktitle={Abhandlungen aus dem Mathematischen Seminar der Universit{\"a}t Hamburg},
  volume={29},
  number={1},
  pages={107--117},
  year={1965},
  organization={Springer}
}

@article{pach1997graphs,
  title={Graphs drawn with few crossings per edge},
  author={Pach, J{\'a}nos and T{\'o}th, G{\'e}za},
  journal={Combinatorica},
  volume={17},
  number={3},
  pages={427--439},
  year={1997},
  publisher={Springer}
}

@article{fujiwara2007note,
  title={A note on spaces of asymptotic dimension one},
  author={Fujiwara, Koji and Whyte, Kevin},
  journal={Algebraic \& Geometric Topology},
  volume={7},
  number={2},
  pages={1063--1070},
  year={2007},
  publisher={Mathematical Sciences Publishers}
}

@article{fujiwara2021asymptotic,
  title={Asymptotic dimension of planes and planar graphs},
  author={Fujiwara, Koji and Papasoglu, Panos},
  journal={Transactions of the American Mathematical Society},
  volume={374},
  number={12},
  pages={8887--8901},
  year={2021}
}

@article{jorgensen2022geodesic,
  title={Geodesic spaces of low {Nagata} dimension},
  author={J{\o}rgensen, Martina and Lang, Urs},
  journal={Annales Fennici Mathematici},
  volume={47},
  pages={83--88},
  year={2022}
}

@article{manning2005geometry,
  title={Geometry of pseudocharacters},
  author={Manning, Jason Fox},
  journal={Geometry \& Topology},
  volume={9},
  number={2},
  pages={1147--1185},
  year={2005},
  publisher={Mathematical Sciences Publishers}
}

\end{document}